\documentclass[11pt]{amsart}

\usepackage{amsmath, amsthm, amssymb, tikz, hyperref, enumerate,array, fullpage,ytableau,shuffle}

\usetikzlibrary{calc}

\usepackage{amsfonts}
\usepackage{fullpage}
\usepackage[enableskew,vcentermath]{youngtab}
\usetikzlibrary{decorations.pathreplacing,patterns}

\theoremstyle{plain}
\newtheorem{theorem}{Theorem}[section]

\newtheorem{lemma}[theorem]{Lemma}
\newtheorem{corollary}[theorem]{Corollary}

\newtheorem{problem}[theorem]{Problem}
\theoremstyle{definition}
\newtheorem{defn}[theorem]{Definition}
\newtheorem{example}[theorem]{Example}

\newtheorem{remark}[theorem]{Remark}
\newtheorem{observation}[theorem]{Observation}

\numberwithin{equation}{section}

\newcommand{\makeset}[2]{ \left\{#1\,|\, #2\right\} }

\newcommand{\ZZ}{{\mathbb {Z}}}

\newcommand{\bij}{\Psi}
\renewcommand{\L}{{\mathcal{L}}}

\newcommand{\Des}{{\operatorname{Des}}}
\newcommand{\cDes}{{\operatorname{cDes}}}

\newcommand{\cdes}{{\operatorname{cdes}}}
\newcommand{\des}{{\operatorname{des}}}

\newcommand{\SYT}{{\operatorname{SYT}}}

\newcommand\red[1]{{\color{red}#1}}

\newcommand{\symm}{{\mathfrak{S}}}

\newcommand{\TTT}{{\mathcal{T}}}

\newcommand{\cA}{{\mathcal{A}}}
\newcommand{\cB}{{\mathcal{B}}}

\newcommand{\cV}{{\mathcal{V}}}
\newcommand{\cH}{{\mathcal{H}}}
\newcommand{\cU}{{\mathcal{U}}}

\newcommand{\cW}{{\mathcal{W}}}

\newcommand{\bx}{ {\bf x} }

\newcommand{\fS}{{\mathfrak{S}}}

\newcommand{\Q}{{\mathcal Q}}
\newcommand{\F}{\mathcal{F}}
\newcommand{\A}{\mathcal{A}}

\newcommand{\compkgamma}{{\mathcal C}^\gamma_{k,n-k}}

\begin{document}
	
	\title[Cyclic Schur-positive sets]
	{On cyclic Schur-positive sets of permutation}

\author{Jonathan Bloom}
\address{Department of Mathematics, Lafayette College, Easton, PA 18042, USA}
\email{bloomjs@lafayette.edu}

\author{Sergi Elizalde}
\address{Department of Mathematics, Dartmouth College, Hanover, NH 03755, USA}
\email{sergi.elizalde@dartmouth.edu}

\author{Yuval Roichman}
\address{Department of Mathematics, Bar-Ilan University, Ramat-Gan 52900, Israel}
\email{yuvalr@math.biu.ac.il}

	\date{\today}

	\thanks{JB was partially supported by a Bar-Ilan University visiting grant. SE was partially supported by Simons Foundation grant \#280575. YR was partially supported by an MIT-Israel MISTI grant and by the Israel Science Foundation, grant no.\ 1970/18. 
}
	
	\maketitle

	\begin{abstract}
We introduce a notion of {\em cyclic Schur-positivity} for sets of permutations, which naturally extends the classical notion of Schur-positivity, and it involves the existence of a bijection from permutations to standard Young tableaux that preserves the cyclic descent set.
Cyclic Schur-positive sets of permutations are always Schur-positive, but the converse does not hold, as exemplified by inverse descent classes, Knuth classes and conjugacy classes. 

In this paper we show that certain classes of permutations invariant under either horizontal or vertical rotation
are cyclic Schur-positive.
The proof unveils a new equidistribution phenomenon of descent sets on permutations, provides affirmative solutions to conjectures from~\cite{ERSchur} and~\cite{AGRR}, and 
yields new examples of Schur-positive sets.

	\end{abstract}

\tableofcontents

\section{Introduction}\label{sec:introduction}

	Let $[n]:=\{1,2,\dots,n\}$ and let $\symm_n$ denote the symmetric group on $[n]$. Recall that the {\em descent set} of a permutation $\pi\in\symm_n$ is
	\[
	\Des(\pi) := \{i\in [n-1]:\ \pi(i)>\pi(i+1)\}. 
	\]

	Given any subset $A\subseteq \symm_n$, we define the
	quasi-symmetric function
\begin{equation}\label{eq:QA}
	\Q(A) := \sum\limits_{\pi\in A} \F_{n,\Des(\pi)},
\end{equation}	
	where $\F_{n,D}$ is Gessel's {\em fundamental quasi-symmetric function}~\cite{Gessel}, defined for $D\subseteq [n-1]$ by
\[
	\F_{n,D}:=\sum\limits_{i_1\le i_2\le \cdots \le i_n \atop
	i_j<i_{j+1} \ \text{if}\ j\in D} x_{i_1}x_{i_2}\cdots x_{i_n}.	    
\]

	A symmetric function is called {\em Schur-positive} if all the
	coefficients in its expansion in the basis of Schur functions are nonnegative.
A subset $A \subseteq \symm_n$ is called {\em Schur-positive} if $\Q(A)$ is symmetric and Schur-positive. 

The following long-standing problem was first addressed in~\cite{Gessel-Reutenauer}.
\begin{problem}
	Find Schur-positive subsets of $\symm_n$. 
\end{problem}

It is possible to characterize Schur-positive permutation sets using standard Young tableaux (SYT). 
We write $\lambda/\mu\vdash n$ to mean that $\lambda/\mu$ is a skew shape with $n$ boxes, where $\lambda$ and $\mu$ are partitions such that the Young diagram of $\mu$ is contained in that of $\lambda$. Let $\SYT(\lambda/\mu)$ denote the set of standard Young tableaux of shape $\lambda/\mu$.  We draw tableaux in English notation, as in Figure~\ref{fig:SYT}. The {\em descent set} of $T \in \SYT(\lambda/\mu)$ is
\begin{equation}\label{eq:defn_Des_SYT} 
\Des(T) := \{i\in [n-1]:\ i+1 \textrm{ is in a lower row than $i$ in $T$}\}.    
\end{equation}
For example, the descent set of the SYT in Figure~\ref{fig:SYT} is $\{1,4,5,8\}$.
	\begin{figure}[htb]
	$$\young(:1348,:257,69)$$
	\caption{A SYT of shape $(5,4,2)/(1,1)$.} 
	\label{fig:SYT}
\end{figure}
For $J\subseteq [n-1]$, let ${\bf x}^J:=\prod\limits_{i\in J}x_i$.

\begin{theorem}[{\cite[Prop. 9.1]{Adin-R}}]\label{thm:Sp-combin}
A subset $A\subseteq \symm_n$ is Schur-positive if and only if there exist nonnegative integers $(m_\lambda)_{\lambda\vdash n}$ such that

\[
\sum\limits_{\pi\in A} {\bf x}^{\Des(\pi)}
    =\sum\limits_{\lambda\vdash n} m_\lambda\sum\limits_{T\in \SYT(\lambda)}{\bf x}^{\Des(T)}. 
\]
\end{theorem}
This characterization of Schur-positive sets of permutations is useful because it does not require computing quasisymmetric functions, but rather finding a $\Des$-preserving bijection from permutations to SYT of shapes given by a certain multiset.

\medskip

In this paper we introduce and study a cyclic analogue of Schur-positive permutation sets, whose definition is motivated by Theorem~\ref{thm:Sp-combin}. Before we state Definition~\ref{def:cSp}, we need some background on cyclic descent sets.

The	{\em cyclic descent set} for permutations was introduced by Cellini~\cite{Cellini} and later studied by Dilks, Petersen, Stembridge~\cite{DPS}, and others. For $\pi\in \symm_n$ let
\begin{equation}\label{eq:cellini}
\cDes(\pi) := \{i\in [n]:\ \pi(i) > \pi(i+1)\},
\end{equation}
with the convention $\pi(n+1):=\pi(1)$.

\begin{example}
	For $\pi=21453\in\symm_5$, we have
		$\Des(\pi)=\{1,4\}$ and
		$\cDes(\pi)=\{1,4,5\}$.
\end{example}

The cyclic descent set for rectangular SYT was introduced by  Rhoades~\cite{Rhoades}, and extended to some other shapes in~\cite{AER,ER16}. This notion was generalized to all skew shapes that are not connected ribbons in~\cite{ARR}.
An explicit combinatorial description of cyclic descent sets on $\SYT(\lambda/\mu)$ for every skew shape which is not a connected ribbon
was recently given by  
Huang~\cite{Huang}. The following is the main definition in this paper.

\begin{defn}
\label{def:cSp}
    A subset $A\subseteq \symm_n$ is {\em cyclic Schur-positive} ({\em cSp}) if 
    there exists a collection of nonnegative integers $(m_{\lambda/\mu})_{\lambda/\mu\, \vdash n}$ such that
\begin{equation}\label{eq:cSp}
    \sum\limits_{\pi\in A} {\bf x}^{\cDes(\pi)}=\sum\limits_{\lambda/\mu\, \vdash n} m_{\lambda/\mu}\sum\limits_{T\in \SYT(\lambda/\mu)}{\bf x}^{\cDes(T)},         
\end{equation}
where $\cDes(\pi)$ is defined by Eq.~\eqref{eq:cellini},
$\cDes(T)$ is the cyclic descent set defined in~\cite{ARR, Huang}, and
the sum in the RHS is over skew shapes $\lambda/\mu$ that are not connected ribbons (and thus cyclic descent sets exist for SYT of these shapes).
\end{defn}

In Theorem~\ref{thm:cSp} we give
an alternative characterization of cyclic Schur-positive permutation sets using an invariance property of cyclic descent sets.

It will follow from this characterization 
that cyclic Schur-positive sets of permutations are always Schur-positive. However, the converse does not hold; in fact,
most known examples of Schur-positive sets are not cyclic Schur-positive.
One of our goals is to address the following problem.

\begin{problem}
Find cyclic Schur-positive (cSp) subsets of $\symm_n$.
\end{problem}

In this paper we present two families of cSp sets of permutations: horizontal rotations of Schur-positive sets, and vertical rotations of inverse descent classes, which include arc permutations. Let us now define these concepts.

To define horizontal and vertical rotations, let $c_n$ denote the $n$-cycle $(1,2,\dots,n)=23\dots n1\in \symm_n$, and let $C_n:=\langle c_n \rangle$ be the cyclic subgroup generated by $c_n$.

Any set $A\subseteq\symm_{n-1}$ can be interpreted as a subset of $\symm_n$ by identifying $\symm_{n-1}$ with the set of permutations in $\symm_n$ that fix $n$. With this interpretation, we define the \emph{horizontal (respectively, vertical) rotation closure} of $A\subseteq\symm_{n-1}$ as the set $AC_n\subseteq \symm_n$ (respectively, $C_nA\subseteq \symm_n$).
Our first main result (Theorem~\ref{thm:horizontal-main1}) states that
if $A \subseteq \symm_{n-1}$ is Schur-positive, then its 
horizontal rotation closure $AC_n\subseteq \symm_n$ is cSp.
As a consequence, we show in Corollary~\ref{cor:cdes} that the set of permutations whose inverses have a given number of cyclic descents is cSp.

For every $J\subseteq [n-2]$, define the  {\em descent class}
\[
D_{n-1,J}:=\{\pi\in \symm_{n-1}:\ \Des(\pi)=J\}.
\]
The inverse of a set of permutations $A\subseteq\symm_n$ is defined as $A^{-1}:=\{\pi^{-1}:\pi\in A\}$.
Our second main result (Theorem~\ref{thm:equid-main1}) implies that for any inverse descent class, the distribution of the statistic $\cDes$ is the same on its vertical rotations $C_n D_{n-1,J}^{-1}$ as on its horizontal rotations $D_{n-1,J}^{-1}C_n$.
The proof of this result, which settles a 
stronger version of~\cite[Conjecture 10.2]{ERSchur},
involves $\cDes$-preserving operations on grid classes, as defined in~\cite{AABRV}.
Even though our proof is not bijective in general, we give explicit $\cDes$-preserving bijections between $D_{n-1,J}^{-1} C_n$ and $C_n D_{n-1,J}^{-1}$ when $|J|=1$ in Section~\ref{sec:singletons}, and when 
$J=[i]$ in Section~\ref{sec:arc}.

Combining the two main theorems mentioned above, it follows that the set $C_n D_{n-1,J}^{-1}$ of vertical rotations of an inverse descent class is cSp (Theorem~\ref{thm:vertical-main}). In particular, it is Schur-positive, as had been conjectured in~\cite{ERSchur}. 

Of special interest is the set $\A_n$ of arc permutations, which are those permutations in $\symm_n$ where every prefix forms an interval in $\ZZ_n$. As we will see in Section~\ref{sec:arc}, this set is a 
union of vertical rotations of inverse descent classes, so we deduce (Corollary~\ref{cor:main_arc}) that it is cSp as well, and we provide a bijective proof of this fact (Theorem~\ref{thm:arcSYT}).

\section{Background}\label{sec:background}

\subsection{Schur-positive permutation sets}
Recall that $A \subseteq \symm_n$ is Schur-positive if the quasi-symmetric function $\Q(A)$ from Equation~\eqref{eq:QA} is symmetric and Schur-positive. 
A direct combinatorial characterization of Schur-positive sets 
was described in Theorem~\ref{thm:Sp-combin}.

Some classical examples of Schur-positive sets of permutations are given in Table~\ref{tab:fine_list}. 
Other examples have appeared more recently in~\cite{ERarc,ERSchur,ER16}.

\begin{table}[h]
    \centering
    \begin{tabular}{p{2.1in}|l|p{2.7in}}
    {\bf Schur-positive subset of $\symm_n$} & {\bf Reference} & {\bf Related examples} \\ \hline 
     Knuth class    & \cite{Gessel} & subsets invariant under Knuth relations (e.g. inverse descent classes, $321$-avoiding permutations) \\ \hline 
     conjugacy class & \cite[Thm.\ 5.5]{Gessel-Reutenauer} & subsets invariant under conjugation
     (e.g. involutions)\\ \hline 
     permutations with a fixed inversion number & \cite[Prop. 9.5]{Adin-R} &
    \end{tabular}
    \caption{Classical examples of Schur-positive sets of permutations}
    \label{tab:fine_list}
\end{table}

\subsection{Cyclic descents for SYT}

Originally introduced by Rhoades~\cite{Rhoades} in the setting of rectangular shapes and later extended to some other shapes in~\cite{AER,Pechenik}, the notion of cyclic descent set of SYT was defined for arbitrary skew shapes in~\cite{ARR}. 
The following definition, introduced in~\cite{AER,ARR}, was motivated by the basic common properties of cyclic descent sets of permutations, defined in Equation~\eqref{eq:cellini}, and of SYT in the cases for which the definition was known at that time.
For a set $D\subseteq[n]$ and an integer $i$, we use the notation $i+D:=\{i+d \bmod n:d\in D\}\subseteq [n]$. Throughout the paper, addition of elements in $[n]$ will be interpreted modulo $n$.

\begin{defn} \label{def:cDes}
		Let $\TTT$ be a finite set. A {\em descent map} is any map
		$\Des: \TTT \longrightarrow 2^{[n-1]}$. 
		A {\em cyclic extension} of $\Des$ (also called a {\em cyclic descent extension}) is
		a pair $(\cDes,\psi)$, where 
		$\cDes: \TTT \longrightarrow 2^{[n]}$ is a map (called a {\em cyclic descent map})
		and $\psi: \TTT \longrightarrow \TTT$ is a bijection,
		satisfying the following axioms:  for all $T$ in  $\TTT$,
		\[
		\begin{array}{rl}
		\text{(extension)}   & \cDes(T) \cap [n-1] = \Des(T),\\
		\text{(equivariance)}& \cDes(\psi(T))  = 1+\cDes(T) ,\\
		\text{(non-Escher)}  & \emptyset \subsetneq \cDes(T) \subsetneq [n].\\
		\end{array}
		\]
\end{defn}

	\medskip

	\begin{theorem}[{\cite[Theorem 1.1]{ARR}}]\label{thm:ARR1}
		Let $\lambda/\mu$ be a skew shape.
		There exists a cyclic descent extension for $\SYT(\lambda/\mu)$ if and only if $\lambda/\mu$ is not a connected ribbon.
	\end{theorem}

An explicit combinatorial description of a cyclic descent extension on $\SYT(\lambda/\mu)$ for every skew shape which is not a connected ribbon
was recently given by Brice Huang~\cite{Huang}.

\begin{example}\label{ex:cdes-SYT} Write $\mu^1\oplus \mu^2 \oplus \cdots \oplus \mu^t$ to denote the skew shape consisting of $t$ connected components $\mu^1,\dots,\mu^t$, ordered from southwest to northeast.
A {\em strip} is a shape  $\mu^1\oplus \mu^2 \oplus \cdots \oplus \mu^t$, 
each of whose connected components has either only one row or only one column. For a SYT $T$ of a strip shape with at least 2 components, let 

$$
\cDes(T)=\begin{cases} \Des(T)\sqcup\{n\} & \text{if $1$ is lower than $n$, or $1$ and $n$ are in the same vertical component,}\\
\Des(T) & \text{otherwise.} \end{cases}
$$

Let $\psi(T):=1+T$, where $j+T$ is the SYT obtained by adding $j$ to each entry of $T$, modulo $n$, then rearranging the letters within each component in increasing order from left to right (if the component is a row) or from top to bottom (if it is a column). Note that $\psi^j(T)=j+T$ for all $j$. As shown in~\cite[Prop. 4.1]{AER}, $\cDes(1+T)=1+\cDes(T)$, so $(\cDes,\psi)$ is a cyclic descent extension.

        For example, letting
    \[
    T=\young(:34,1,2), \quad 1+T=\young(:14,2,3), \quad 2+T=\young(:12,3,4), \quad 3+T=\young(:23,1,4),
    \]
   the corresponding cyclic descent sets are $\{1,4\}$, $\{1,2\}$, $\{2,3\}$ and $\{3,4\}$, respectively.

In the special case where SYT is a horizontal strip shape with at least 2 components, we have
    \[
    \cDes(T):=\{i\in [n]:\  i+1 \b
     n \ \text{is in a lower row than $i$ in $T$}\}.
    \]

For example, letting
    \[
    T=\young(::34,12), \quad 1+T=\young(::14,23), \quad 2+T=\young(::12,34), \quad 3+T=\young(::23,14),
    \]
    the corresponding cyclic descent sets are $\{4\}$, $\{1\}$, $\{2\}$ and $\{3\}$, respectively.
\end{example}

\begin{example}\label{ex2:cdes-SYT}

By~\cite[Corollary 3.9]{AER}, for every $0\le k<n-1$, 
there  exists a unique cyclic descent map on $\SYT((n-1-k,1^k)\oplus (1))$, specified by letting 
\begin{equation}\label{eq:cDes_near_hook}
\cDes(T)=\begin{cases} \Des(T)\sqcup\{n\} & \text{if $|\Des(T)|=k$,}\\
\Des(T) & \text{otherwise,} \end{cases}
\end{equation} 
for every $T\in \SYT((n-1-k,1^k)\oplus (1))$. 
Denote by $\delta(T)$ the entry in the northeast disconnected cell of such tableau.

For $0<j<n$, let
$j+T$ be the SYT obtained by letting the set of entries in the first column (excluding the corner cell) be $j+1+\left(\cDes(T)\setminus \delta(T)\right)$ if $n-j\not\in \Des(T)$,
and $j+1+\left(\cDes(T)\setminus\{n-j\}\right)$ if  $n-j\in \Des(T)$, and letting $\delta(j+T):=j+\delta(T)$.  By~\cite[Lemma 4.7]{AER}, $\cDes(j+T)=j+\cDes(T)$, so a cyclic descent extension is obtained by taking $\psi(T):=1+T$.
\end{example}

The following lemma is used in Section~\ref{sec:cSp}.

\begin{lemma}[{\cite[Lemma 2.2]{ARR}}]\label{lem:ARR1}
Let $\TTT$ be a finite set which carries a descent map $\Des: \TTT \longrightarrow 2^{[n-1]}$. Assume that there exists a cyclic descent extension $(\cDes,\psi)$.
Then, for every subset $\emptyset \ne J = \{j_1 < \cdots < j_t\} \subseteq [n]$, the fiber size $|\cDes^{-1}(J)|$ is uniquely determined by the formula
\begin{equation}
\label{cDes-fiber-sizes-formula}
|\cDes^{-1}(J)|
= \sum_{i=1}^{t} (-1)^{i-1} |\Des^{-1}(\{j_{i+1} - j_i, \ldots, j_t - j_i\})|,
\end{equation}
where we interpret $\{j_{i+1} - j_i, \ldots, j_t - j_i\}$ as $\emptyset$ when $i=t$.
\end{lemma}

\section{Cyclic Schur-positive permutation sets}\label{sec:cSp}

\subsection{Basic examples}
Definition~\ref{def:cSp} introduces the concept of cyclic Schur-positivity (cSp), which is the central notion of this paper. 
Let us make some remarks about this definition. First, note the analogy with Theorem~\ref{thm:Sp-combin}, which characterizes Schur-positivity of sets of permutations. 
One difference, however, is the use of skew shapes in Definition~\ref{def:cSp}. This modification is needed because SYT of hook shapes do not carry a cyclic descent extension. 

Second, note that Lemma~\ref{lem:ARR1} 
implies that, for every skew shape $\lambda/\mu$ which is not a connected ribbon, 
the polynomial $$\sum\limits_{T\in \SYT(\lambda/\mu)}{\bf x}^{\cDes(T)}$$ is well defined, in the sense that it does not depend on the choice of $\cDes$. 

\begin{example}\label{cSp-examples}
\begin{enumerate}[1.]

\item
Recall the cyclic subgroup $C_n$ generated by the $n$-cycle $c_n=(1,2,\dots,n)$.
By the definitions in Example~\ref{ex:cdes-SYT}, we have
\[
\sum_{\pi\in C_n} {\bf x}^{\cDes(\pi)}
=\sum\limits_{i=1}^n x_i
=\sum\limits_{T\in \SYT((1)\oplus(n-1))}{\bf x}^{\cDes(T)},
\]
thus $C_n$ is cSp.

\item By~\cite[Theorem 1.2]{ARR},
		\[
		\sum_{\pi \in \symm_n} {\bf x}^{\cDes(\pi)}
		= \sum_{\substack{\text{non-hook}\\ \lambda \vdash n}}
		|\SYT(\lambda)| \sum_{T \in \SYT(\lambda)} {\bf x}^{\cDes(T)}
		\quad + \quad \sum_{k=1}^{n-1} \binom{n-2}{k-1} 
		\sum_{T \in \SYT((n-k+1,1^k)/(1))} {\bf x}^{\cDes(T)},
		\]
		thus $\symm_n$ is cSp. 
\end{enumerate}
\end{example}

\subsection{An alternative characterization}
The following equivariance property will be used to give an alternative characterization of cSp sets.

\begin{defn}\label{def:cyc-invariant}  
A subset $A\subseteq \symm_n$ is called {\em $\cDes$-invariant} if there exists a bijection $\psi:A\to A$ such that, for all $\pi\in A$,
\begin{equation}\label{eq:cyc-invariant}
\cDes(\psi \pi)=1+\cDes(\pi).
\end{equation}
\end{defn}

\begin{remark}\label{rem:rotation} 
A subset $A\subseteq \symm_n$ is {\em invariant under horizontal rotation} if
$A=Ac_n$, where $c_n= (1,2,\dots,n)$. 
Subsets invariant under horizontal rotation are $\cDes$-invariant, since one can simply take $\psi\pi=\pi c_n^{-1}$ in Definition~\ref{def:cyc-invariant}.
\end{remark}

The following characterization of cSp sets will be used in the rest of the paper. 

\begin{theorem}\label{thm:cSp}	 
A subset $A\subseteq \symm_n$ is cyclic Schur-positive if  and only if it is  Schur-positive and $\cDes$-invariant.
\end{theorem}

Before proving this theorem, let us consider some examples of its usage. 

\begin{example}\label{ex:cSp-alt}
\begin{enumerate}[1.]
\item The set $C_n=\langle c_n \rangle$ satisfies 
\[
\sum\limits_{\pi \in C_n} {\bf x}^{\Des(\pi)}=1+\sum\limits_{i=1}^{n-1}x_i=\sum\limits_{T\in \SYT((1)\oplus(n-1))}{\bf x}^{\Des(\pi)}=\sum\limits_{T\in \SYT(n-1,1)\cup\SYT(n)}{\bf x}^{\Des(\pi)}.
\]
Thus, by Theorem~\ref{thm:Sp-combin}, it is Schur-positive. On the other hand, by Remark~\ref{rem:rotation},  $C_n$ is $\cDes$-invariant.
Hence, by Theorem~\ref{thm:cSp}, $C_n$ is cSp, in agreement with Example~\ref{cSp-examples}.1.
   
\item The Knuth class $K=\{2143,2413\}\subset \symm_4$, corresponding to the tableau 
$\begin{ytableau}
\scriptstyle 1& \scriptstyle3\\
\scriptstyle 2&\scriptstyle 4
\end{ytableau}$, is Schur-positive, since
\[
\sum\limits_{\pi\in K} {\bf x}^{\Des(\pi)}=x_1 x_3+x_2=\sum\limits_{Q\in \SYT(2,2)} {\bf x}^{\Des(Q)}.
\]
However, $K$ is not $\cDes$-invariant, because $\cDes(2143)=\{1,3,4\}$ and $\cDes(2413)=\{2,4\}$. Thus, it is not cSp.
\end{enumerate}
\end{example}

\begin{remark}\label{rem:most}
 By Theorem~\ref{thm:cSp}, every cSp set of permutations is Schur-positive. The converse does not hold, as Example~\ref{ex:cSp-alt}.2 shows. 
\end{remark}

We now turn our attention to the proof of Theorem~\ref{thm:cSp}. In the next two lemmas, we assume that $A\subseteq\symm_n$ is Schur-positive and $\cDes$-invariant, and we let $(m_{\lambda})_{\lambda\vdash n}$ be 
nonnegative integers such that
\begin{equation}\label{eq:111}
    \sum\limits_{\pi\in A} {\bf x}^{\Des(\pi)}=\sum\limits_{\lambda\vdash n} m_\lambda\sum\limits_{T\in \SYT(\lambda)}{\bf x}^{\Des(T)},
\end{equation}
which exist by Theorem~\ref{thm:Sp-combin}.

\begin{lemma}\label{lem:1}
For every $0\le k <n$,
\[
m_{(n-k,1^k)}=|\{a\in A:\ \Des(a)=[k]\}|,
\]
where we define $[0]:=\emptyset$.
\end{lemma}
\begin{proof}
For every $0\le k <n$, there is a unique SYT $T$ with $\Des(T) = [k]$, namely the tableau of shape $(n-k,1^k)$ having $1,\ldots, k+1$ in the first column and $1,k+2,\ldots, n$ in the first row. Comparing the coefficients of ${\bf x}^{[k]}$ on both sides of Equation~\eqref{eq:111} completes the proof.
\end{proof}

\begin{lemma}\label{lem:2}
For every $0\le k <n$, the alternating sum
\[
\sum\limits_{i=k}^{n-1} (-1)^{k-i} m_{(n-i,1^i)}
\]
is nonnegative; when $k=0$, it is zero.
\end{lemma}

\begin{proof}
To simplify notation, let us first write, for any $S\subseteq [n]$,
$$A_S:=\makeset{a\in A}{\cDes(a)=S}.$$
By Lemma~\ref{lem:1}, for $0\le k <n$,
$$m_{(n-i,1^i)}=|\{a\in A:\ \Des(a)=[i]\}|=|A_{[i]}| + |A_{[i]\cup \{n\}}|.$$
Since $A$ is $\cDes$-invariant, we have
$|A_{[i]\cup \{n\}}| = |A_{[i+1]}|$.  Combining these facts we see that
\[\sum_{i=k}^{n-1} (-1)^{k-i} m_{(n-i,1^i)}=\sum_{i=k}^{n-1} (-1)^{k-i} (|A_{[i]}|+|A_{[i]\cup \{n\}}| )=\sum_{i=k}^{n-1} (-1)^{k-i} (|A_{[i]}| + |A_{[i+1]}|).
\]
The telescoping sum on the right-hand side further reduces to
\[
|A_{[k]}|+(-1)^{k-n+1}|A_{[n]}|=|A_{[k]}|,
\]
using that $A_{[n]} = \emptyset$ by the non-Escher property.
This expression is clearly nonnegative, and it equals zero when $k=0$, since $A_{[0]} = \emptyset$ again by the non-Escher property.
\end{proof}

\begin{lemma}\label{lem:3}
    Let $A\subseteq \fS_n$ be $\cDes$-invariant.  Assume that there are nonnegative constants $(m_{\lambda/\mu})_{\lambda/\mu\, \vdash n}$ such that
    \begin{equation}\label{eq:lem3}
    \sum\limits_{\pi\in A} {\bf x}^{\Des(\pi)}=\sum\limits_{\lambda/\mu\, \vdash n} m_{\lambda/\mu}\sum\limits_{T\in \SYT(\lambda/\mu)}{\bf x}^{\Des(T)}.         
    \end{equation}
    Further, assume that  $m_{\lambda/\mu} = 0$ whenever $\lambda/\mu$ is a connected ribbon.
     Then $A$ is cSp.  
\end{lemma}

\begin{proof}
Let $B(\bx)$ be the quantity in Equation~\eqref{eq:lem3}.  As no ribbon shape appears with positive multiplicity on the right-hand side, we know from Theorem~\ref{thm:ARR1} that a cyclic descent extension $(\cDes, \psi)$ exists so that \[
D(\bx):= \sum\limits_{\lambda/\mu\, \vdash n} m_{\lambda/\mu}\sum\limits_{T\in \SYT(\lambda/\mu)}{\bf x}^{\cDes(T)}
\]
is a well-defined quantity. Set 
$$C(\bx) := \sum_{\pi\in A}\bx^{\cDes(\pi)},$$
where $\cDes$ is now the cyclic descent for permutations given by Equation~\eqref{eq:cellini}.
   
By definition, this cyclic descent for permutations satisfies the extension and non-Escher properties from Definition~\ref{def:cDes}. In addition, since $A$ is $\cDes$-invariant, there exists a bijection $\rho :A\to A$ through which $\cDes$ satisfies the equivariance property. Thus, the pair $(\cDes,\rho)$ is a cyclic descent extension on $A$.

For every subset $\emptyset \ne J = \{j_1 < \cdots < j_t\} \subseteq [n]$, let $b(J)$, $c(J)$ 
and $d(J)$ be the coefficients of $\bx^J$ in $B(\bx)$, $C(\bx)$ and $D(\bx)$, respectively.      
It follows from Lemma~\ref{lem:ARR1} that
\[
c(J)
= \sum_{i=1}^{t} (-1)^{i-1} b(\{j_{i+1} - j_i, \ldots, j_t - j_i\})= d(J).
\]
Thus $C(\bx)=D(\bx)$, which implies that $A$ is cSp.
\end{proof}

\begin{proof}[Proof of Theorem~\ref{thm:cSp}]

For $0\leq k\leq n$, we have
\begin{equation}\label{eq:pf1}
\sum\limits_{T\in \SYT(1^k\oplus (n-k))}{\bf x}^{\Des(T)}
=\sum\limits_{T\in \SYT(n-k,1^k)}{\bf x}^{\Des(T)}+\sum\limits_{T\in \SYT(n-k+1,1^{k-1})}{\bf x}^{\Des(T)}.   
\end{equation}
Indeed, each tableau in $\SYT(1^k\oplus (n-k))$ where the leftmost value in the first row is smaller (respectively, larger) than the top value in the first column corresponds to a tableau in $\SYT(n-k,1^k)$ (respectively, $\SYT(n-k+1,1^{k-1})$).  Solving for the first term on the right-hand side and iterating the resulting equality, we get 
\begin{equation}\label{eq:pf111}
\sum_{T\in \SYT(n-k,1^k)} {\bf x}^{\Des(T)}
=\sum_{i=0}^{k} (-1)^{k-i} \sum_{T\in \SYT(1^{i}\oplus (n-i))} {\bf x}^{\Des(T)}.
\end{equation}

Now assume that $A$ is Schur-positive and $\cDes$-invariant, and let $m_\lambda$ be given by Equation~\eqref{eq:111}.
By Equality~\eqref{eq:pf111}, 
the contribution from hook shapes to the right-hand side of Equation~\eqref{eq:111} is
\begin{align*}
\sum\limits_{k=0}^{n-1} m_{(n-k,1^k)}\sum_{T\in \SYT(n-k,1^k)} {\bf x}^{\Des (T)} & =
\sum\limits_{k=0}^{n-1} m_{(n-k,1^k)}\sum_{i=0}^{k} (-1)^{k-i} \sum_{T\in \SYT(1^i\oplus (n-i))} {\bf x}^{\Des(T)}\\
&=\sum_{i=0}^{n-1} d_i \sum_{T\in \SYT(1^i\oplus (n-i))} {\bf x}^{\Des(T)},
\end{align*}
interchanging the order of summation and letting $\displaystyle d_i:=\sum\limits_{k=i}^{n-1}(-1)^{k-i} m_{(n-k,1^{k})}$. By Lemma~\ref{lem:2}, the coefficients $d_0,\ldots, d_{n-1}$ are nonnegative integers and $d_0= 0$.    One can now rewrite Equation~\eqref{eq:111} as 
\[
\sum\limits_{\pi\in A} {\bf x}^{\Des(\pi)} =\sum\limits_{\lambda\vdash n\atop\lambda\ \text{not a hook}}m_{\lambda}\sum_{T\in \SYT(\lambda)}{\bf x}^{\Des(T)} +
\sum_{i=1}^{n-2} d_i \sum_{T\in \SYT(1^i\oplus (n-i))} {\bf x}^{\Des(T)}.
\]
Observe that none of the shapes on the right-hand side are ribbon shapes.  It now follows from Lemma~\ref{lem:3} that $A$ is cSp.

\medskip

For the converse, assume now that $A$ is cSp. Setting $x_n=1$ in Equation~\eqref{eq:cSp} gives
\[
\sum\limits_{\pi\in A} {\bf x}^{\Des(\pi)}=\sum\limits_{\lambda/\mu\vdash n} m_{\lambda/\mu}\sum\limits_{T\in \SYT(\lambda/\mu)}{\bf x}^{\Des(T)}. 
    \]
Applying the vector space isomorphism from the multilinear subspace of the formal power series ring $\ZZ[x_1,x_2,\ldots]$
to the ring of quasisymmetric functions, defined by ${\bf x}^J\mapsto \F_{n,J}$, 
we get
\[
\Q(A)=\sum\limits_{\lambda/\mu\vdash n} m_{\lambda/\mu}\sum\limits_{T\in \SYT(\lambda/\mu)}\F_{n,\Des(T)}
=\sum\limits_{\lambda/\mu\vdash n} m_{\lambda/\mu}s_{\lambda/\mu},
\]
where the last equality uses Gessel's identity~\cite[Theorem 7.19.7]{EC2}.

Using the Littlewood-Richardson rule, which expresses skew Schur functions as non-negative linear combinations of Schur functions~\cite[Eq. (A1.142)]{EC2}, we see that $A$ is Schur-positive.

Finally, by Equation~\eqref{eq:cSp}, $A$ is $\cDes$-invariant because so is the corresponding collection of SYT of the skew shapes given by the right-hand side.
\end{proof}

\subsection{Horizontal rotations}\label{sec:hor}
We conclude this section with some applications of Theorem~\ref{thm:cSp} to sets of permutations that are invariant under horizontal rotation.
The next result follows immediately from  Theorem~\ref{thm:cSp} together with Remark~\ref{rem:rotation}.  

	\begin{theorem}\label{cor:horizontal2}
If $A\subseteq \symm_n$ is Schur-positive and invariant under horizontal rotation, then it is cSp.
 	 \end{theorem}

Another consequence is the fact that horizontal rotation closures of Schur-positive sets are cSp.

\begin{theorem}\label{thm:horizontal-main1}
	For every Schur-positive set $A \subseteq \symm_{n-1}$, the set $AC_n\subseteq \symm_n$ is cSp. 
\end{theorem}

\begin{proof}[Proof of Theorem~\ref{thm:horizontal-main1}]
For every Schur-positive set $A \subseteq \symm_{n-1}$, the set $AC_n\subseteq \symm_n$ is Schur-positive by~\cite[Theorem 1.1]{ER16}. 
Since $AC_n$ is invariant under horizontal rotation, Theorem~\ref{cor:horizontal2} implies that it is cSp.
\end{proof}

Define the {\em cyclic descent number} of $\pi$ as $\cdes(\pi):= |\cDes(\pi)|$.

	\begin{corollary}\label{cor:cdes}
For every $n> k\ge 1$, the set 
\[
C_{n,k}:=\{\pi\in \symm_n:\ \cdes(\pi^{-1})=k\}
\]
is cSp.
	\end{corollary}

\begin{proof}
It is shown in~\cite[Corollary 7.7]{ERSchur} that $C_{n,k}$
is Schur-positive. By~\cite[Lemma 6.4]{ERSchur}, it is invariant under horizontal rotation.  
Thus, by Theorem~\ref{cor:horizontal2}, it is cSp.
\end{proof}		

A more transparent, self-contained proof of Corollary~\ref{cor:cdes} will be given in Section~\ref{sec:applications_thm:equid-main}.

\section{Vertical versus horizontal rotations}\label{hor-vert}

In this section we prove the following equidistribution result, and we discuss applications of it.

\begin{theorem}\label{thm:equid-main1}
	For every $J\subseteq [n-2]$,
	\[
	\sum\limits_{\pi\in C_n D_{n-1,J}^{-1}}{\bf x}^{\cDes(\pi)}t^{\pi^{-1}(n)} =
	\sum\limits_{\pi\in D_{n-1,J}^{-1}C_n}{\bf x}^{\cDes(\pi)}t^{\pi^{-1}(n)}.
	\]
\end{theorem}

Our proof is not bijective, but explicit bijections for the special cases when  $J=\{j\}$ (singletons) and
$J=[i]$ (prefixes) will be given in Sections~\ref{sec:singletons} and~\ref{sec:arc}, respectively.

\subsection{Proof of Theorem~\ref{thm:equid-main1}}
We first recall some basic definitions.  A  {\em composition} (resp.\ {\em weak composition}) of $n\geq 0$ is a finite sequence of positive (resp.\ non-negative) integers
$\gamma=(\gamma_1,\dots,\gamma_t)$ whose sum is $n$.

For a positive integer $n$, there is a natural bijection from  
subsets of $[n-1]$ to compositions of $n$. Indeed, to the subset $J=\{j_1<j_2<\cdots<j_{t-1}\}\subseteq [n-1]$ we associate
the composition  $\gamma=(\gamma_1, \dots,\gamma_{t})$ of $n$  
defined as follows: set $j_0=0$ and $j_{t}=n$, and let
$\gamma_i=j_i-j_{i-1}$,
for $1\leq i\leq t$. 
Further, define
\begin{equation*}
 S(\gamma):=\bigsqcup_{I\subseteq J}D_{n,I}^{-1},
\end{equation*}
where $\sqcup$ denotes disjoint union. 
In other words, $S(\gamma)$ is the set of permutations where, for each $1\le i\le t$, the entries $j_{i-1}+1,j_{i-1}+2,\dots,j_i$ appear from left to right in increasing order.
At times we also denote this set by $S(\gamma_1,\ldots,\gamma_t)$ when consideration of the parts in $\gamma$ is required.  Refining this set, let
$$S(\gamma_1,\ldots, \gamma_i^*,\ldots,\gamma_t):=\makeset{\pi \in  S(\gamma)}{\pi(n) = j_i}$$
and  
$$S(\gamma_1,\ldots, {}_{*}\gamma_i,\ldots,\gamma_t):=\makeset{\pi \in  S(\gamma)}{\pi(1) = j_{i-1} +1},$$
so that
\begin{equation*}
\bigsqcup_{i\in[t]}S(\gamma_1,\ldots, \gamma_i^*,\ldots,\gamma_t) = S(\gamma) = \bigsqcup_{i\in [t]}S(\gamma_1,\ldots, {}_{*}\gamma_i,\ldots,\gamma_t).	
\end{equation*}

For the remainder of this section, fix a composition $\gamma=(\gamma_1,\ldots,\gamma_t)$ of $n$ and denote by $J\subseteq[n-1]$ its corresponding subset.   The relevance of these definitions to the theorem at hand is the fact that 
\begin{equation}\label{eq:CD=CS}
\bigsqcup_{I\subseteq J} C_nD_{n-1,I}^{-1} = C_n S(\gamma_1,\ldots, \gamma_t^*)\quad \textrm{and} \quad \bigsqcup_{I\subseteq J} D_{n-1,I}^{-1}C_n = S(\gamma_1,\ldots, \gamma_t^*)C_n.
\end{equation}
As we shall see, it is easier to first argue in terms of such unions and then, via an application of inclusion-exclusion, conclude our desired result. To keep track of the position of the largest letter, we introduce two further refinements.  Let
\begin{equation}\label{eq:defV}
    \cV_{\gamma}^k=\{\pi\in C_n S(\gamma_1,\dots,\gamma_t^*): \pi(k)=n\},
\end{equation}
so that $\cV_{\gamma}^k$ is the result of vertically rotating the permutations in $S(\gamma_1,\dots,\gamma_t^*)$ until the largest value is in the $k$th position.  Likewise, define
\begin{equation}\label{eq:defH}
\cH_{\gamma}^k=\{\pi\in  S(\gamma_1,\dots,\gamma_t^*)C_n: \pi(k)=n\},
\end{equation}
so that this set is the result of horizontally rotating the same set of permutations until the largest letter is again in the $k$th position.  We can now state our main technical lemma from which the proof of Theorem~\ref{thm:equid-main1} follows almost immediately.  

\begin{lemma}\label{lem:VH}
For every $k\leq n$, there exists a $\cDes$-preserving bijection between $\cV_{\gamma}^k$ and $\cH_{\gamma}^k$.
\end{lemma}

Assuming this lemma for the moment, let us prove the theorem.

\begin{proof}[Proof of Theorem~\ref{thm:equid-main1}]
By Lemma~\ref{lem:VH}, it follows that
\[
\sum\limits_{\pi\in C_n S(\gamma_1,\ldots,\gamma_t^*)}{\bf x}^{\cDes(\pi)}t^{\pi^{-1}(n)} =
\sum\limits_{\pi\in S(\gamma_1,\ldots,\gamma_t^*)C_n}{\bf x}^{\cDes(\pi)}t^{\pi^{-1}(n)}.
\]
With $J\subseteq [n-1]$ defined as above, it follows from (\ref{eq:CD=CS}) and the principle of inclusion-exclusion that 
\[
C_nD_{n-1,J}^{-1}=\sum\limits_{I\subseteq J} (-1)^{|J\setminus I|} C_nS(\gamma_1,\ldots, \gamma_t^*).
\]
Additionally, the analogous equality involving right multiplication by $C_n$ also holds.  Together, these facts yield the desired statement.
\end{proof}

To prove Lemma~\ref{lem:VH}, we recall the notion of a shuffle. Let $\iota_n:= 12\ldots n$ denote the increasing permutation.  For any nonnegative integers $a$ and $b$, define 
$$\iota_a \shuffle \iota_b$$
to be the set of permutations in $\symm_{a+b}$ where the letters in $[a]$ appear from left to right in increasing order, and so do the letters in $[a+b]\setminus [a]$.
By definition,
$$S(\gamma_1,\ldots, \gamma_t) = \bigsqcup_{I\subseteq J}D_{n,I}^{-1}=  \iota_{\gamma_1} \shuffle\cdots \shuffle \iota_{\gamma_t}.$$

The next lemma belongs to mathematical folklore, but we include a proof for the sake of completeness, as well as some examples.

\begin{lemma}\label{lem:shuffleab}
For every $a,b>0$, there exists a $\Des$-preserving bijection
$$\varphi: \iota_a \shuffle \iota_b \to \iota_b \shuffle \iota_a.$$
\end{lemma}

\begin{proof}
Permutations $\pi\in\iota_a \shuffle \iota_b$ can be encoded bijectively as words $w$ over a binary alphabet $\{1,2\}$ with $a$ $1$s and $b$ $2$s, 
where $w_i=1$ if $\pi_i\le a$, and $w_i=2$ if $\pi_i>a$.

Given $\pi\in\iota_a \shuffle \iota_b$, let $w$ its corresponding binary word. By splitting $w$ at the descents of $\pi$, which correspond to occurrences of $21$ in $w$, we can write $$w=1^{i_1} 2^{j_1}\,2|1\,1^{i_2} 2^{j_2}\,2|1\dots 2|1\,1^{i_k} 2^{j_k},$$
with $i_r,j_r\ge0 $ for all $r$. 
Let 
$$f(w)=1^{j_1} 2^{i_1}\,2|1\,1^{j_2} 2^{i_2}\,2|1\dots 2|1\,1^{j_k} 2^{i_k},$$
and define $\varphi(\pi)$ to be the permutation in $\iota_b \shuffle \iota_a$ encoded by $f(w)$. 
By construction, $\Des(\varphi(\pi))=\Des(\pi)$.
\end{proof}

\begin{example} 
Let $\pi=1\,2\,3\,8\,9 \, 4\,10\,11 \, 5\,6\,12 \, 7\in\iota_7\shuffle\iota_5$, which has $\Des(\pi)=\{5,8,11\}$. Encoding $\pi$ as a binary word and splitting at the descents, we obtain
$w=1112\,2|1\,2\,2|1\,1\,2|1$. Applying $f$ from the second proof of Lemma~\ref{lem:shuffleab}, we get the word 
$f(w)=1222\,2|1\,1\,2|1\,2\,2|1$, which encodes the permutation $\varphi(\pi)=1\,6\,7\,8\,9\,2\,3\,10\,4\,11\,12\,5\in\iota_5\shuffle\iota_7$.
\end{example}

It is possible to extend the construction in Lemma~\ref{lem:shuffleab} to shuffles of $t\ge2$ increasing sequences.
Let $a_1,\dots,a_t,n$ be positive integers such that $a_1+\dots+a_t=n$.  We can explicitly construct a bijection
$$\varphi: \iota_{a_1} \shuffle  \dots \shuffle \iota_{a_t} \to \iota_{a_t} \shuffle  \dots \shuffle \iota_{a_1}$$
as follows. 

For a given word $u$ over the alphabet $\{1,\dots,t\}$ and $1\le j<t$, define $f_{j}(u)$ to be the word obtained by 
fixing in place all the entries of $u$ that are not equal to $j$ or $j+1$, and applying the map $f$ from the second proof of Lemma~\ref{lem:shuffleab} to the subword of $u$ consisting of the entries $j$ and $j+1$ (ignoring the other entries).

Given $\pi\in\iota_{a_1} \shuffle  \dots \shuffle \iota_{a_t}$, we can encode it as a word $w$ of length $n$ over the alphabet $\{1,\dots,t\}$ with $a_j$ $j$s for $1\le j\le t$, by letting $w_i=j$ if $a_1+\dots+a_{j-1}<\pi_i\le a_1+\dots+a_{j}$, for all $i$.

Fix a reduced decomposition of the decreasing permutation $t\dots21$, say $t\dots21=s_{j_1}s_{j_2}\dots s_{j_{r}}$ (where $r=\binom{t}{2}$), and apply the map $f_{s_{j_1}}\circ f_{s_{j_2}}\circ\dots\circ f_{s_{j_r}}$ to $w$.
Let $\varphi(\pi)$ be the permutation in $\iota_{a_t} \shuffle  \dots \shuffle \iota_{a_1}$ encoded by the resulting word.

\begin{example} 
Let $\pi= 1\, 6\, 7\, 2\, 12\, 3\, 8\, 9\, 13\, 10\, 4\, 14\, 5\, 11\in \iota_5\shuffle\iota_6\shuffle\iota_3$, which 
is encoded by the word $w=12213122321312$. Note that $\Des(\pi)=\{3,5,9,10,12\}$. Fixing the reduced decomposition $321=s_1 s_2 s_1$, we get
$$\begin{array}{ccccc}
w=122|1\red{3}122\red{3}2|1\red{3}12\quad\stackrel{f_1}{\mapsto} & 122|1\red{3}112\red{3}2|1\red{3}12 &&& \\
& \red122\red13\red1\red1|23|2\red13\red1|2 &\stackrel{f_2}{\mapsto} & \red133\red13\red1\red1|23|2\red13\red1|2 & \\
&&& 1\red3\red31\red3112\red32|1\red312 &\stackrel{f_1}{\mapsto}\quad1\red3\red32\red3222\red32|1\red312,
\end{array}$$
and so $\varphi(\pi)= 1\, 10\, 11\, 4\, 12\, 5\, 6\, 7\, 13\, 8\, 2\, 14\, 3\, 9\in  \iota_3\shuffle\iota_6\shuffle\iota_5$.
\end{example}

The definition provides a ``horizontal'' decomposition of permutations that arise from shuffles.  

\begin{defn}\label{def:circledast}
Let $\alpha$ and $\beta$ be weak compositions, each with $t$ parts.  Further, assume that $\gamma = \alpha+\beta$, where addition is componentwise.  For $\rho \in S(\alpha_1,\ldots,\alpha_t)$ and $\sigma \in S(\beta_1,\ldots, \beta_t)$, define 
$$\rho\circledast \sigma$$ 
to be the unique permutation in $S(\gamma_1,\ldots, \gamma_t)$ whose leftmost $|\alpha|$ entries are order-isomorphic to $\rho$ and whose rightmost $|\beta|$ entries are order-isomorphic to $\sigma$.  
\end{defn}

A pictorial representation of this construction when $t=3$ is given in Figure~\ref{fig:circledast}. In particular, if $\alpha = (2,2,0)$, $\beta = (3,1,2)$, and we take $1342\in S(\alpha)$ and $561423\in S(\beta)$, then
$$1\,3\,4\,2 \circledast 5\,6\,1\,4\,2\,3 = 1\,6\,7\,2\,9\,10\,3\,8\,4\,5.$$

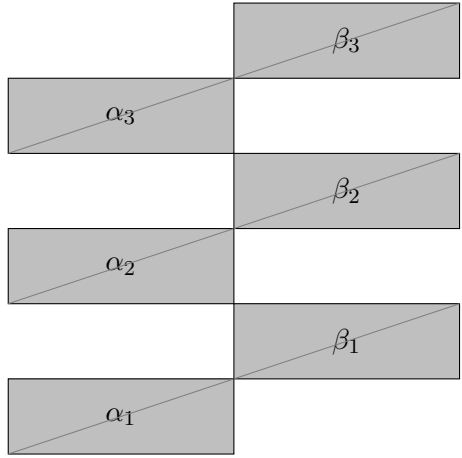
\begin{figure}[htb] \centering
\begin{tikzpicture}
\draw[fill = lightgray](0,0) rectangle (3,1) node[midway] {$\alpha_1$}; \draw[gray] (0,0)--(3,1);
\draw[fill = lightgray](0,2) rectangle (3,3) node[midway] {$\alpha_2$}; \draw[gray] (0,2)--(3,3);
\draw[fill = lightgray](0,4) rectangle (3,5) node[midway] {$\alpha_3$}; \draw[gray] (0,4)--(3,5);

\draw[fill = lightgray](3,1) rectangle (6,2) node[midway] {$\beta_1$}; \draw[gray] (3,1)--(6,2);
\draw[fill = lightgray](3,3) rectangle (6,4) node[midway] {$\beta_2$}; \draw[gray] (3,3)--(6,4);
\draw[fill = lightgray](3,5) rectangle (6,6) node[midway] {$\beta_3$}; \draw[gray] (3,5)--(6,6);
	\end{tikzpicture}
\caption{A visualization for $t=3$ of $\rho\circledast \sigma$ as given by Definition~\ref{def:circledast}. We place (maintaining their positions) the values in $\sigma$ corresponding to $\iota_{\alpha_i}$ in the box labeled $\alpha_i$, and similarly for $\rho$.}
\label{fig:circledast}
\end{figure}

An immediate consequence of this construction is that, for every fixed $k\leq n$, we have
$$S(\gamma_1,\ldots,\gamma_t) = \bigsqcup_{(\alpha,\beta)\in\compkgamma} S(\alpha_1,\ldots,\alpha_t)\circledast S(\beta_1,\ldots,\beta_t),$$
where $\compkgamma$ denotes the set of pairs of weak compositions $\alpha$ of $k$ and $\beta$ of $n-k$, each with $t$ parts, such that $\alpha+\beta= \gamma$.  

For our next lemma, recall that $c_n=(1,2,\ldots, n)$, and that  $\cV_\gamma^k$ was defined in Equation~\eqref{eq:defV}.  
\begin{lemma}\label{lem:V}
For every $1\le k\le n$,
$$\cV_\gamma^kc_n^{k} = 	
	\bigsqcup_{(\alpha,\beta)\in\compkgamma}\bigsqcup_{i\in[t]} 
S(\beta_i,\ldots, \beta_t^*,\beta_1\ldots,\beta_{i-1})
 \circledast 
S(\alpha_{i+1},\ldots, \alpha_t, \alpha_1,\ldots, \alpha_i^*).$$
\end{lemma}

\begin{proof}
Observe that 
$$S(\gamma_1,\ldots, \gamma_t^*) = \bigsqcup_{(\alpha,\beta)\in\compkgamma} S(\alpha_1,\ldots, \alpha_t) \circledast S(\beta_1,\ldots,\beta_t^*).$$
Refining this set by the value in the $k$th position, we get
\begin{equation}\label{eq:decomp S(gamma)}
S(\gamma_1,\ldots, \gamma_t^*) = \bigsqcup_{(\alpha,\beta)\in\compkgamma}\bigsqcup_{i\in [t]} S(\alpha_1,\ldots,\alpha_i^*,\ldots, \alpha_t) \circledast S(\beta_1,\ldots,\beta_t^*).    
\end{equation}
To obtain our claim, vertically rotate each permutation in \eqref{eq:decomp S(gamma)} so that the value in the $k$th position is largest, and then horizontally rotate the resulting permutation $k$ positions to the left.
\end{proof}

The decomposition in the previous lemma motivates the next result.

 \begin{lemma}\label{lem:simplification}
There is a $\cDes$-preserving bijection between
\begin{equation}\label{eq:cdp1}
\bigsqcup_{i\in[t]} 
S(\beta_i,\ldots,\beta_t^*,\beta_1,\ldots, \beta_{i-1})
 \circledast 
S(\alpha_{i+1},\ldots, \alpha_t, \alpha_1,\ldots, \alpha_i^*)	
\end{equation}
and 
\begin{equation}\label{eq:cdp2}
S(\beta_t-1,\beta_{t-1},\ldots, \beta_1)
 \circledast 
S(\alpha_{t},\alpha_{t-1},\ldots,\alpha_2,(\alpha_1+1)^*).	
\end{equation}
 \end{lemma}
\begin{proof}
Since all the permutations $\pi$ in the sets \eqref{eq:cdp1} and \eqref{eq:cdp2} are such that $\pi_n=n$, it suffices to prove the existence of a $\Des$-preserving bijection between 
\[
\cU:=\bigsqcup_{i\in [t]} 
S(\beta_i,\ldots,\beta_t^*,\beta_1,\ldots, \beta_{i-1})
 \circledast 
S(\alpha_{i+1},\ldots, \alpha_t, \alpha_1,\ldots, \alpha_i-1)
\]
and 
\[
\cW:=S(\beta_t-1,\beta_{t-1},\ldots, \beta_1)
 \circledast 
S(\alpha_{t},\ldots,\alpha_1),
\]
i.e., the sets obtained from (\ref{eq:cdp1}) and (\ref{eq:cdp2}) by deleting this largest final value.  
We use the notation $\mathcal{A}\sim\mathcal{B}$ to denote that there exists a $\Des$-preserving bijection between the sets $\mathcal{A}$ and $\mathcal{B}$. To show that $\cU\sim\cW$, we will construct a sequence of $\Des$-preserving bijections.

Using Lemma~\ref{lem:shuffleab} repeatedly,  we can reorder $\beta_1,\ldots, \beta_{i-1}$ and, separately, reorder $\beta_i, \ldots, \beta_t^*$ to obtain
\begin{equation}\label{eq:cdp5}
S(\beta_i,\ldots,\beta_t^*, \beta_1, \ldots,  \beta_{i-1})
\sim 
S(\beta_t-1,\beta_{t-1},\ldots,(\beta_i+1)^*,  \ldots,  \beta_1).	
\end{equation}
As the permutations on both sides of Equation~\eqref{eq:cdp5} all have the same value $\beta_i +\cdots +\beta_t$ in the last position, it follows that 
$$\cU \sim \bigsqcup_{i\in [t]} 
S(\beta_t-1,\beta_{t-1},\ldots,(\beta_i+1)^*,  \ldots,  \beta_1)
 \circledast 
S(\alpha_{i+1},\ldots, \alpha_t, \alpha_1,\ldots, \alpha_i-1).$$
When $i =1$, the set inside the disjoint union becomes
\begin{align*}
&S(\beta_t-1,\beta_{t-1},\ldots, (\beta_1+1)^*)
 \circledast 
S(\alpha_{2},\ldots, \alpha_t, \alpha_1-1)\\
&\sim S(\beta_t-1,\beta_{t-1},\ldots, \beta_1)
 \circledast 
S(\alpha_{2},\ldots, \alpha_t, {}_{*}\alpha_1)\\
&\sim 
S(\beta_t-1,\beta_{t-1},\ldots, \beta_1)
 \circledast 
S(\alpha_t,\dots, \alpha_2, {}_{*}\alpha_1),
\end{align*}
and when $i\neq 1$, we have
\begin{align*}
&\bigsqcup_{i\neq 1} 
S(\beta_t-1,\beta_{t-1},\ldots, \beta_1)
 \circledast 
S(\alpha_{i+1},\ldots, \alpha_t, {}_{*}(\alpha_1+1),\ldots, \alpha_i-1)\\
& \sim \bigsqcup_{i\neq 1} 
S(\beta_t-1,\beta_{t-1},\ldots, \beta_1)
 \circledast 
S(\alpha_t,\ldots, {}_{*}\alpha_i,\ldots, \alpha_1),
\end{align*}
where the equivalence follows by interchanging the order of the $\beta$'s just as we did above with the $\alpha$'s.  Summing over $i\in[t]$ shows that there exists a $\Des$-preserving bijection between $\cU$ and  $\cW$, proving our lemma.
\end{proof}

Armed with these technical lemmas, we prove Lemma~\ref{lem:VH}.

\begin{proof}[Proof of Lemma~\ref{lem:VH}]
For any two sets $\cA,\cB\subseteq \fS_n$, there exists a $\cDes$-preserving bijection between $\cA$ and $\cB$ if and only if there is one between $\cA c_n^{k}$ and $\cB c_n^{k}$.  
Thus, to prove the lemma, it suffices to show that such a bijection exists between $\cH_\gamma^kc_n^k$ and $\cV_\gamma^kc_n^k$.  By the definitions involved, we have
$$\cH_\gamma^kc_n^k = S(\gamma_1,\ldots,\gamma_t^*).$$
On the other hand, by Lemmas~\ref{lem:V} and \ref{lem:simplification}, there is a $\cDes$-preserving bijection between $\cV_\gamma^kc_n^k$ and
$$\bigsqcup_{(\alpha,\beta)\in\compkgamma} S(\beta_t-1, \beta_{t-1},\ldots, \beta_1) 
 \circledast
S(\alpha_{t},\ldots, (\alpha_1+1)^*) = S(\gamma_t-1,\gamma_{t-1},\ldots,\gamma_2,  (\gamma_1+1)^*).$$
Applying Lemma~\ref{lem:shuffleab} to the right-hand side and noting that all permutations end with the largest letter $n$, we get a $\cDes$-preserving bijection between this set and $S(\gamma_1,\ldots, \gamma_t^*)$.  This completes our proof. 
\end{proof}

\subsection{Consequences of Theorem~\ref{thm:equid-main1}}\label{sec:applications_thm:equid-main}
The following is one of the main applications of the theorem. 

\begin{theorem}\label{thm:vertical-main}
 For every $J\subseteq [n-2]$, the vertically rotated inverse descent class $C_n D_{n-1,J}^{-1}\subseteq \symm_n$ is cSp.
\end{theorem}

\begin{proof}
As shown in~\cite{Gessel} (see also Table~\ref{tab:fine_list}), inverse descent classes are Schur-positive. Hence, by Theorem~\ref{thm:horizontal-main1}, for every
$J\subseteq [n-2]$, the horizontally rotated inverse descent class $D_{n-1,J}^{-1} C_n$
is cSp.
Finally, by Theorem~\ref{thm:equid-main1}, the distribution  of $\cDes$ on 
 $C_n D_{n-1,J}^{-1}$ is the same as on $D_{n-1,J}^{-1} C_n$, completing the proof.   
\end{proof}

Next we apply Theorem~\ref{thm:vertical-main} 
to give a self-contained proof of Corollary~\ref{cor:cdes}, which states that the set of permutations with a given inverse cyclic descent number is cSp.

\begin{proof}[Second proof of Corollary~\ref{cor:cdes}]
        Note that 
        \[
        \{\pi\in \symm_n:\ \cdes(\pi)=k\}=\bigsqcup_{J\subseteq [n-2]\atop |J|=k-1} D_{n-1,J}C_n.
        \]
        Thus
        \[
        C_{n,k}=\{\pi\in \symm_n:\ \cdes(\pi^{-1})=k\}=\bigsqcup_{J\subseteq [n-2]\atop |J|=k-1} C_n D_{n-1,J}^{-1}.
        \]
        Theorem~\ref{thm:vertical-main} completes the proof.
\end{proof}

The set $C_{n,k}$ had been shown to be Schur-positive in~\cite[Corollary 7.7]{ERSchur}. 
A significantly stronger version of this result is conjectured in~\cite[Conjecture 
7.1]{AGRR}. Originally formulated in the language of cyclic quasi-symmetric functions, cyclic compositions and cyclic permutations, this conjecture equivalent to the statement that, for every $\emptyset \subsetneq J \subsetneq [n]$, the set  of permutations in $\symm_n$ whose inverse has cyclic descent set of the form $i+J$ for some $i$ (with addition modulo $n$ as usual) is Schur-positive. 
The following consequence of Theorem~\ref{thm:vertical-main} provides an affirmative solution to this conjecture.

\begin{corollary}\label{conj1} 
	For every $\emptyset \subsetneq J \subsetneq [n]$, the set 
	\begin{equation}\label{eq:iplusJ}
	\{\pi \in \symm_n:\ \cDes(\pi^{-1})=i+J \textrm{ for some }i\}
	\end{equation}
	is cSp.
	In particular, it is Schur-positive.
\end{corollary}

\begin{proof}
For $j\in J$, let $\Delta_j:=D_{n-1,-j+(J\setminus\{j\})}$, viewed as a subset of $\symm_n$, and 
let $A$ denote the set in Equation~\eqref{eq:iplusJ}.

Given $\pi\in\symm_n$, we have that $\cDes(\pi)=J$ if and only if there exists $j\in J$ such that $j-1\notin J$ and $\pi=\sigma c_n^j$ for some $\sigma\in \Delta_j$. Indeed, the `if' direction is clear, and the other is obtained by taking $j=\pi^{-1}(n)$.
It follows that $\cDes(\pi)=i+J$ if and only if there exists $j\in J$ such that $j-1\notin J$ and $\pi=\sigma c_n^{j+i}$ for some $\sigma\in \Delta_j$. Letting $i$ vary, we get
$$A^{-1}=\bigcup_{j\in J,\atop j-1\notin J}\Delta_j C_n.$$

Any two sets in this union are either equal or disjoint. Specifically, 
$\Delta_{j_1} C_n\cap\Delta_{j_2} C_n=\emptyset$ unless $-j_1+J=-j_2+J$, in which case $\Delta_{j_1} C_n=\Delta_{j_2} C_n$. It follows that we can express $A^{-1}$ as a disjoint union of sets $\Delta_j C_n$ for $j$ in some subset $J'\subseteq J$. Taking inverses, we get 
$$A=\bigsqcup_{j\in J'}C_n \Delta_j^{-1},$$
which is cSp by	Theorem~\ref{thm:vertical-main}, and thus Schur-positive by Theorem~\ref{thm:cSp}.
\end{proof}

\begin{example}
\begin{itemize}
\item 
Let $n=4$ and $J=\{2,4\}$. Then $-2+J\setminus \{2\}=-4+J\setminus\{4\}=\{2\}$, so
$$\{\pi \in \symm_4: \cDes(\pi^{-1})=\{2,4\} \ \text{or}\ \{1,3\} \}=C_4 D_{3,\{2\}}^{-1}.$$ 

\item 
Let $n=6$ and $J=\{1,3,5\}$. Then $-1+J\setminus\{1\}=-3+J\setminus\{3\}=-5+J\setminus\{5\}=\{2,4\}$, so 
$$\{\pi \in \symm_6: \cDes(\pi^{-1})=\{1,3,5\}\ \text{or}\ \{2,4,6\} \}=C_6 D_{5,\{2,4\}}^{-1}.$$

\item 
Let $n=5$ and $J=\{1,2,4\}$. Then $-1+J\setminus\{1\}=\{1,3\}$ and $-4+J\setminus\{4\}=\{2,3\}$, so
		$$\{\pi \in \symm_5:\ \cDes(\pi^{-1})=i+\{1,2,4\} \text{ for some }i\}
		=C_5 D^{-1}_{4, \{1,3\}}\sqcup C_5 D^{-1}_{4, \{2,3\}}.$$ 
\end{itemize}
\end{example}

\medskip

Our last application of Theorem~\ref{thm:equid-main1} 
  is an affirmative solution to \cite[Conjecture 10.2]{ERSchur}. 

\begin{corollary}[{\cite[Conjecture 10.2]{ERSchur}}]\label{conj:CD}
\begin{itemize}
    \item[1.]  
    For every  $J\subseteq [n-2]$, the distribution of $\Des$ over the 
	vertically rotated inverse descent class $C_n D_{n-1,J}^{-1}$ is
	the same as over $D_{n-1,J}^{-1} C_n$.
	\item[2.] The set $C_n D_{n-1,J}^{-1}$ is Schur-positive.
\end{itemize}
\end{corollary}

\begin{proof}
To prove part 1, let $x_n=t=1$ in  Theorem~\ref{thm:equid-main1}.
For part 2, we know that $C_n D_{n-1,J}^{-1}$ is cSp by Theorem~\ref{thm:vertical-main}, and thus it is Schur-positive by Theorem~\ref{thm:cSp}.
\end{proof}

\section{An explicit bijection for shuffles of two increasing sequences} \label{sec:singletons}

In this section and the next one we describe explicit bijections proving Theorem~\ref{thm:equid-main1} in two special cases.
This section deals with the case that $|J|=1$, and 
Section~\ref{sec:arc} will deal with the case $J=[i]$.

In the rest of this section, suppose that $J$ contains one element, that is, $J=\{j\}$. In this case, $D_{n-1,J}^{-1}$ consists of shuffles of two increasing sequences of fixed length. Specifically, 
$D_{n-1,J}^{-1}=S(j,n-j^*)\setminus\{12\dots n\}$.
We give an explicit bijection $$\bij:C_nD_{n-1,\{j\}}^{-1}\to D_{n-1,\{j\}}^{-1}C_n$$ that preserves $\cDes$ and the position of $n$.
This bijection is obtained by analyzing the proof of Theorem~\ref{thm:equid-main1}, and using the map $\varphi$ in the 
proof of Lemma~\ref{lem:shuffleab} to make the bijection explicit.

For a word $w$ over the alphabet $\{1,2\}$, recall from the proof of Lemma~\ref{lem:shuffleab} that 
$f(w)$ is obtained by keeping the descents $21$ untouched, and changing each consecutive block of letters of $w$ that is not part of a $21$, which must be of the form $1^r 2^s$, into $1^s 2^r$. Additionally, for $p<q$, define $f_{[p,q]}(w)$ to be the word obtained by applying the above operation $f$ to the factor $w_{p}w_{p+1}\dots w_q$, and  keeping the other entries of $w$ unchanged.

Next we describe how to obtain $\bij(\pi)$ for given $\pi\in C_n D_{n-1,\{j\}}^{-1}$.
Let $k=\pi^{-1}(n)$ be the position of $n$. If $k=n$, we simply define $\bij(\pi)=\pi$. Suppose now that $k\neq n$.
Rotate $\pi$ horizontally $k$ positions to the left, so that its rightmost entry becomes $n$. The resulting permutation $\pi c_n^k$ is a shuffle of two increasing sequences, so it can be encoded as a word $w$ over $\{1,2\}$ by placing $1$s in the positions corresponding to the lower increasing sequence, and $2$s elsewhere.
If $w_{n-k}=1$, let $w'=f_{[n-k+1,n-1]}(w)$, else (that is, if $w_{n-k}=2$) let $w'=f_{[1,n-k-1]}(w)$.
Let $w''=f_{[1,n-1]}(w')$, and let $\sigma$ be the permutation in $D_{n-1,\{j\}}^{-1}$ (viewed as a subset of $\symm_n$) encoded by $w''$.
Let $\bij(\pi)=\sigma c_n^{-k}$, that is, the permutation obtained by rotating $\sigma$ horizontally $k$ positions to the right.
A schematic description of this construction is given in Figure~\ref{fig:bij}.

\def\eps{.1}

\begin{figure}
    \centering
\begin{tikzpicture}[scale=.5]
	\draw[thin,dotted] (0,3)--(6,3);
	\draw[thin,dotted] (2,0)--(2,6);
	\draw (0,0) rectangle (6,6);
	\draw[thick] (0,0)--(2,1) node[midway,above]{$\beta_1$};
	\draw[thick] (2,1)--(6,3) node[midway,above]{$\alpha_1$};
	\draw[thick] (0,3)--(2,4) node[midway,above]{$\beta_2$};
	\draw[thick] (2,4)--(6,6) node[midway,above]{$\alpha_2$};
	\draw[fill] (6-\eps,6-\eps) circle (\eps);
	\draw (3,8) node {$D_{n-1,\{j\}}^{-1}$};
    \draw [decorate,decoration={brace,amplitude=5pt},xshift=0pt,yshift=-2pt] (2,0) -- (0,0) node [midway,below,yshift=-2pt] {$k$};

	\draw (8,3) node {$\longrightarrow$};
	\draw (8,3.5) node {vertical};
	\draw (8,2.5) node {rotation};

	\begin{scope}[shift={(10,0)}]
	\draw[thin,dotted] (0,2)--(6,2);
	\draw[thin,dotted] (0,5)--(6,5);
	\draw[thin,dotted] (2,0)--(2,6);
	\draw (0,0) rectangle (6,6);
	\draw[thick] (0,2)--(2,3) node[midway,above]{$\beta_1$};
	\draw[thick] (0,5)--(2,6) node[midway,above]{$\beta_2$};
    \draw[thick] (2,3)--(6,5) node[midway,above]{$\alpha_1$};
	\draw[thick] (2,0)--(6,2) node[midway,above]{$\alpha_2$};
    \draw[fill] (2-\eps,6-\eps) circle (\eps);
    \draw[fill] (6-\eps,2-\eps) circle (\eps);
	\draw (7,8) node {$C_nD_{n-1,\{j\}}^{-1}$ with $\pi(k)=n$};
	\draw (3,6.7) node {case 1};
	\draw (7,3) node {$\sqcup$};
	\draw [decorate,decoration={brace,amplitude=5pt},xshift=0pt,yshift=-2pt] (2,0) -- (0,0) node [midway,below,yshift=-2pt] {$k$};
	\end{scope}

\begin{scope}[shift={(18,0)}]
	\draw[thin,dotted] (0,2)--(6,2);
	\draw[thin,dotted] (0,5)--(6,5);
	\draw[thin,dotted] (2,0)--(2,6);
	\draw (0,0) rectangle (6,6);
	\draw[thick] (0,2)--(2,3) node[midway,above]{$\beta_2$};
	\draw[thick] (0,5)--(2,6) node[midway,above]{$\beta_1$};
    \draw[thick] (2,3)--(6,5) node[midway,above]{$\alpha_2$};
	\draw[thick] (2,0)--(6,2) node[midway,above]{$\alpha_1$};
    \draw[fill] (2-\eps,6-\eps) circle (\eps);
    \draw[fill] (6-\eps,5-\eps) circle (\eps);
	\draw (3,6.7) node {case 2};
	\draw [decorate,decoration={brace,amplitude=5pt},xshift=0pt,yshift=-2pt] (2,0) -- (0,0) node [midway,below,yshift=-2pt] {$k$};
	\end{scope}
\end{tikzpicture}
\medskip

\begin{tikzpicture}[scale=.5]

	\draw[thin,dotted] (0,2)--(6,2);
	\draw[thin,dotted] (0,5)--(6,5);
	\draw[thin,dotted] (2,0)--(2,6);
	\draw (0,0) rectangle (6,6);
	\draw[thick] (0,2)--(2,3) node[midway,above]{$\beta_1$};
	\draw[thick] (0,5)--(2,6) node[midway,above]{$\beta_2$};
    \draw[thick] (2,3)--(6,5) node[midway,above]{$\alpha_1$};
	\draw[thick] (2,0)--(6,2) node[midway,above]{$\alpha_2$};
    \draw[fill] (2-\eps,6-\eps) circle (\eps);
    \draw[fill] (6-\eps,2-\eps) circle (\eps);
    \draw (3,8) node {$\pi\in C_nD_{n-1,\{j\}}^{-1}$};
    \draw[->] (6,8)--(33,8)--(33,-17.7);
    \draw (31.5,8.5) node {$\bij$};
    \draw (3,6.7) node {case 1};
	\draw (7,3) node {$\stackrel{\cdot c_n^k}{\rightarrow}$};
	\draw [decorate,decoration={brace,amplitude=5pt},xshift=0pt,yshift=-2pt] (2,0) -- (0,0) node [midway,below,yshift=-2pt] {$k$};
	
\begin{scope}[shift={(8,0)}]
	\draw[thin,dotted] (0,3)--(6,3);
	\draw[thin,dotted] (4,0)--(4,6);
	\draw (0,0) rectangle (6,6);
	\draw[thick] (0,0)--(4,2) node[midway,above]{$\alpha_2$};
	\draw[thick] (4,2)--(6,3) node[midway,above]{$\beta_1$};
	\draw[thick] (0,3)--(4,5) node[midway,above]{$\alpha_1$};
	\draw[thick] (4,5)--(6,6) node[midway,above]{$\beta_2$};
	\draw[fill] (6-\eps,6-\eps) circle (\eps);
	\draw[fill] (4-\eps,2-\eps) circle (\eps);
	\draw (7.75,3) node {$\stackrel{f_{[n-k+1,n-1]}}{\rightarrow}$};
	\draw [decorate,decoration={brace,amplitude=5pt},xshift=0pt,yshift=-2pt] (6,0) -- (4,0) node [midway,below,yshift=-2pt] {$k$};
\end{scope}

\begin{scope}[shift={(17.5,0)}]
	\draw[thin,dotted] (0,3)--(6,3);
	\draw[thin,dotted] (4,0)--(4,6);
	\draw (0,0) rectangle (6,6);
	\draw[thick] (0,0)--(4,2) node[midway,above]{$\alpha_2$};
	\draw[thick] (4,2)--(6,3) node[midway,above]{$\beta_2-1$};
	\draw[thick] (0,3)--(4,5) node[midway,above]{$\alpha_1$};
	\draw[thick] (4,5)--(6,6) node[midway,above]{$\beta_1+1$};
	\draw[fill] (6-\eps,6-\eps) circle (\eps);
	\draw[fill] (4-\eps,2-\eps) circle (\eps);
	\draw (3,-1.2) node {$\sqcup$};
	\draw [decorate,decoration={brace,amplitude=5pt},xshift=0pt,yshift=-2pt] (6,0) -- (4,0) node [midway,below,yshift=-2pt] {$k$};
\end{scope}

\begin{scope}[shift={(0,-8.5)}]
	\draw[thin,dotted] (0,2)--(6,2);
	\draw[thin,dotted] (0,5)--(6,5);
	\draw[thin,dotted] (2,0)--(2,6);
	\draw (0,0) rectangle (6,6);
	\draw[thick] (0,2)--(2,3) node[midway,above]{$\beta_2$};
	\draw[thick] (0,5)--(2,6) node[midway,above]{$\beta_1$};
    \draw[thick] (2,3)--(6,5) node[midway,above]{$\alpha_2$};
	\draw[thick] (2,0)--(6,2) node[midway,above]{$\alpha_1$};
    \draw[fill] (2-\eps,6-\eps) circle (\eps);
    \draw[fill] (6-\eps,2-\eps) circle (\eps);
    \draw (3,6.7) node {case 2};
	\draw (7,3) node {$\stackrel{\cdot c_n^k}{\rightarrow}$};
	\draw [decorate,decoration={brace,amplitude=5pt},xshift=0pt,yshift=-2pt] (2,0) -- (0,0) node [midway,below,yshift=-2pt] {$k$};
	
\begin{scope}[shift={(8,0)}]
	\draw[thin,dotted] (0,3)--(6,3);
	\draw[thin,dotted] (4,0)--(4,6);
	\draw (0,0) rectangle (6,6);
	\draw[thick] (0,0)--(4,2) node[midway,above]{$\alpha_1$};
	\draw[thick] (4,2)--(6,3) node[midway,above]{$\beta_2$};
	\draw[thick] (0,3)--(4,5) node[midway,above]{$\alpha_2$};
	\draw[thick] (4,5)--(6,6) node[midway,above]{$\beta_1$};
	\draw[fill] (6-\eps,6-\eps) circle (\eps);
	\draw[fill] (4-\eps,5-\eps) circle (\eps);
	\draw (7.75,3) node {$\stackrel{f_{[1,n-k-1]}}{\rightarrow}$};
	\draw [decorate,decoration={brace,amplitude=5pt},xshift=0pt,yshift=-2pt] (6,0) -- (4,0) node [midway,below,yshift=-2pt] {$k$};
\end{scope}

\begin{scope}[shift={(17.5,0)}]
	\draw[thin,dotted] (0,3)--(6,3);
	\draw[thin,dotted] (4,0)--(4,6);
	\draw (0,0) rectangle (6,6);
	\draw[thick] (0,0)--(4,2) node[midway,above]{$\alpha_2-1$};
	\draw[thick] (4,2)--(6,3) node[midway,above]{$\beta_2$};
	\draw[thick] (0,3)--(4,5) node[midway,above]{$\alpha_1+1$};
	\draw[thick] (4,5)--(6,6) node[midway,above]{$\beta_1$};
	\draw[fill] (6-\eps,6-\eps) circle (\eps);
	\draw[fill] (4-\eps,5-\eps) circle (\eps);
	\draw (3,-1) node {$\shortparallel$};
	\draw [decorate,decoration={brace,amplitude=5pt},xshift=0pt,yshift=-2pt] (6,0) -- (4,0) node [midway,below,yshift=-2pt] {$k$};
\end{scope}
\end{scope}

\begin{scope}[shift={(17.5,-16.5)}]
	\draw[thin,dotted] (0,3)--(6,3);
	\draw (0,0) rectangle (6,6);
	\draw[thick] (0,0)--(6,3) node[midway]{$\alpha_2+\beta_2-1$};
	\draw[thick] (0,3)--(6,6) node[midway]{$\alpha_1+\beta_1+1$};
	\draw[fill] (6-\eps,6-\eps) circle (\eps);
	\draw (7.5,3) node {$\stackrel{f_{[1,n-1]}}{\rightarrow}$};

\begin{scope}[shift={(8.7,0)}]
	\draw[thin,dotted] (0,3)--(6,3);
	\draw (0,0) rectangle (6,6);
	\draw[thick] (0,0)--(6,3) node[midway]{$\alpha_1+\beta_1$};
	\draw[thick] (0,3)--(6,6) node[midway]{$\alpha_2+\beta_2$};
	\draw[fill] (6-\eps,6-\eps) circle (\eps);
	\draw (3,-1) node {${\downarrow}${\footnotesize $\cdot c_n^{-k}$}};
	\draw (4.5,-2.2) node {$\bij(\pi)\in D_{n-1,\{j\}}^{-1}C_n$};
\end{scope}
\end{scope}
\end{tikzpicture}
    \caption{A schematic description of the map $\bij$. Permutations $\pi\in C_n D_{n-1,\{j\}}^{-1}$ (obtained by vertically rotating $D_{n-1,\{j\}}^{-1}$) with $\pi(k)=n$ fall into two cases. Here $\alpha_1+\beta_1=j$, $\alpha_2+\beta_2=n-j$, and $\beta_1+\beta_2=k$.}
    \label{fig:bij}
\end{figure}
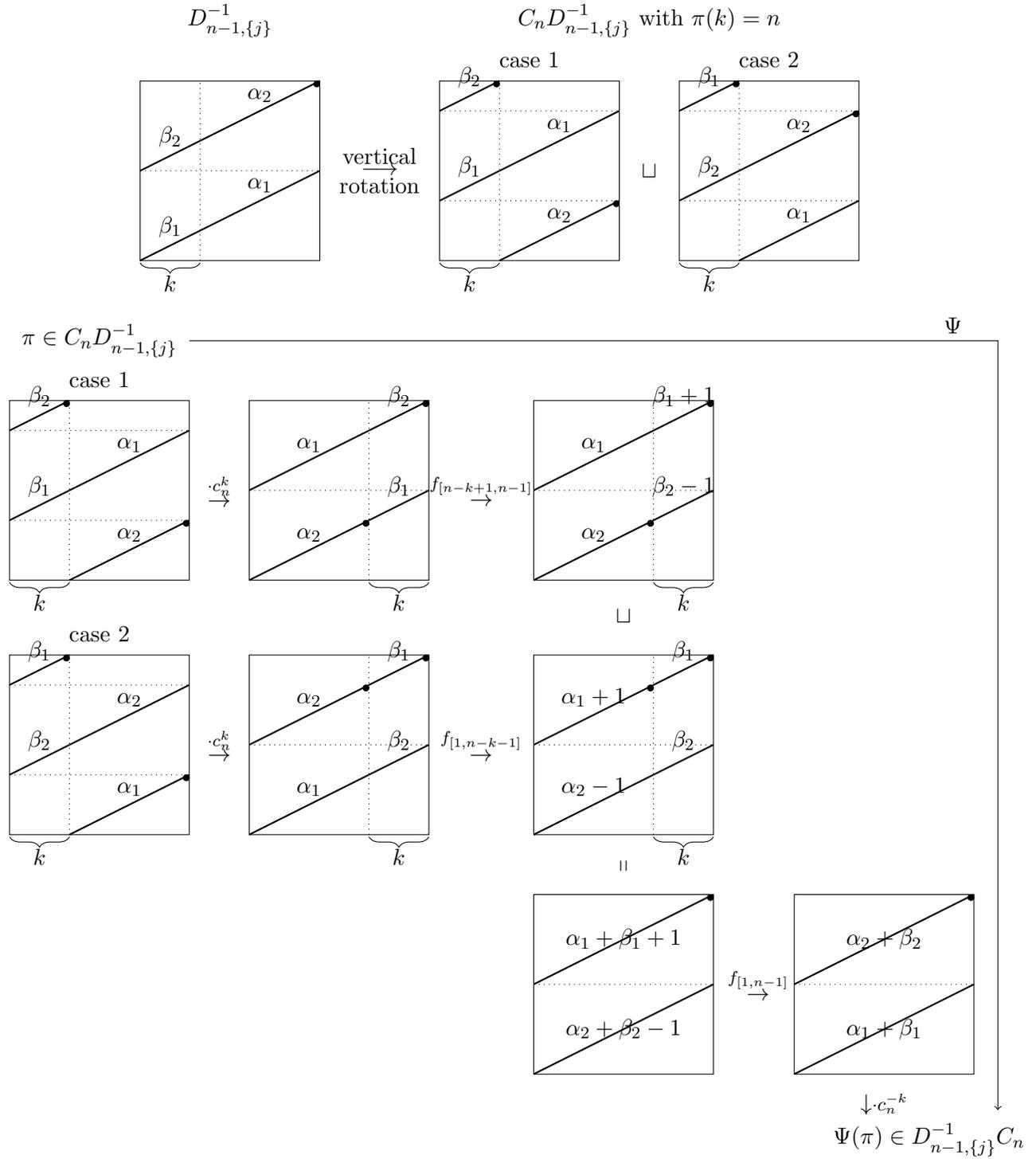

\begin{example}
Let $n=16$ and $J=\{7\}$. Let $\pi = 13\, 6\, 14\, 15\, 7\, 16\, 1\, 8\, 9\, 10\, 2\, 3\, 4\, 11\, 12\, 5\in C_n D_{n-1,J}^{-1}$, which is a vertical rotation of $8\, 1\, 9\, 10\, 2\, 11\, 12\, 3\, 4\, 5\, 13\, 14\, 15\, 6\, 7\, 16\in D_{n-1,J}^{-1}$ (viewed as a subset of $\symm_n$). Since $k=\pi^{-1}(16)=6$, we first compute $\pi c_n^6= 1\, 8\, 9\, 10\, 2\, 3\, 4\, 11\, 12\, 5\, 13\, 6\, 14\, 15\, 7\, 16$, which is encoded by the word $w=1222111221212212$. Since $w_{10}=1$, we define 
$$w'=f_{[11,15]}(1222111221\,21221\,2)= 1222111221\,21121\,2,$$
and $$w''=f_{[1,15]}(122211122121121\,2) = 112211222121221\,2.$$
This word encodes the permutation $\sigma=1\, 2\, 8\, 9\, 3\, 4\, 10\, 11\, 12\, 5\, 13\, 6\, 14\, 15\, 7\, 16$, and so
$\Psi(\pi)= \sigma c_n^{-6}= 13\, 6\, 14\, 15\, 7\, 16\, 1\, 2\, 8\, 9\, 3\, 4\, 10\, 11\, 12\, 5$.
\end{example}

\section{Arc permutations}\label{sec:arc}

In this section we give an explicit bijective proof of Theorem~\ref{thm:equid-main1} in the case $J=[i]$. 
Unlike the bijection $\bij$ from Section~\ref{sec:singletons}, this bijection does not follow from our proof of Theorem~\ref{thm:equid-main1}.
We will see that the theorem, in this special case, becomes a statement about arc permutations.

\begin{defn}
	A permutation $\pi\in \symm_n$ is an {\em arc permutation}
	if, for every $1\le j\le n$, the first $j$ letters in $\pi$ form
	an interval in $\ZZ_n$. Denote by $\A_n$ the set of arc permutations in $\symm_n$.
\end{defn}

For example, $12543\in\A_5$, but $125436\notin\A_6$, since
$\{1,2,5\}$ is an interval in $\ZZ_5$ but not in $\ZZ_6$.

Arc permutations were introduced in the study of flip graphs of polygon triangulations. It was shown in~\cite{ERarc} that they can be characterized as those avoiding the eight patterns $\sigma\in\symm_4$ with $|\sigma(1)-\sigma(2)|=2$, that is,
$$\A_n=\symm_n(1324, 1342, 2413, 2431, 3124, 3142, 4213, 4231).$$
Other combinatorial properties of these permutations, such as their descent set distribution, are studied in~\cite{ERarc}. In particular, it follows from~\cite[Theorem 5]{ERarc}
that $\A_n$ is a Schur-positive set. 
Notice that while $c_n\A_n=\A_n$ for every $n$, we have that $\A_n c_n\ne \A_n$ for $n>3$. In other words, $\A_n$ is invariant under vertical rotation but
not under horizontal rotation, and thus Theorem~\ref{cor:horizontal2} does not apply. However, we will show that it is possible to express arc permutations in terms of vertical rotations of inverse descent classes.

\subsection{Permutation classes and grids}\label{sec:grids}

\begin{defn}
	A sequence of integers $a_1,\dots,a_n$ is 
	\begin{itemize}
	\item {\em left-unimodal} if it is a union of an increasing subsequence and a decreasing subsequence, which intersect at the first letter $a_1$;	
	\item {\em right-unimodal} 
	if the sequence  $a_n,a_{n-1},\dots,a_1$ is left-unimodal.
	\end{itemize}

	\noindent
	We say that a permutation $\pi\in \symm_n$ has one of the above properties if the sequence $\pi(1), \pi(2), \dots, \pi(n)$ does.
	Denote the sets of left-unimodal and right-unimodal permutations in $\symm_n$ by $\L_n$ and ${\mathcal{R}}_n$, respectively.  Note that for every $n\ge 1$, $\pi\in \L_n$ (respectively, $\pi\in {\mathcal{R}}_n$) if and only if, for every $1\le j\le n$, the first (respectively, last) $j$ letters in $\pi$ form an interval in $\ZZ$. These sets can be described in terms of pattern avoidance as $\L_{n}=\symm_{n}(132,312)$ and ${\mathcal{R}}_{n}=\symm_{n}(231,213)$, respectively.
\end{defn}	

Arc permutations are obtained as vertical rotations of 
left-unimodal permutations, that is, $\A_n=C_n \L_{n-1}$. Using that
\begin{equation}\label{eq:L-union}
\L_{n-1}=\bigsqcup_{i=0}^{n-2}D_{n-1,[i]}^{-1},
\end{equation}
we can express $\A_n$
as a disjoint union of vertically rotated inverse descent classes:
\begin{equation}\label{eq:A-union}
\A_n=\bigsqcup_{i=0}^{n-2} C_n D_{n-1,[i]}^{-1}.
\end{equation}
The next result follows now from Theorem~\ref{thm:vertical-main}.

\begin{corollary}\label{cor:main_arc}
The set $\A_n$ is cSp.	
\end{corollary}

Next we recall the notion of geometric grid classes from~\cite{AABRV}, which will be useful for our bijection in Subsection~\ref{sec:bijection}.

\begin{defn}
A {\em geometric grid class} consists of those permutations that can be drawn on a specified set of
line segments of slope $\pm 1$, arranged according to the positions of the non-zero entries of a matrix $M$ with entries in $\{0, +1, -1\}$.
Specifically, a permutation $\pi\in\symm_n$ can be drawn on these line segments if $n$ points can be placed on them so that the $i$th point from the left is the $\pi(i)$th point from the bottom, for all $1\le i\le n$. 
\end{defn}	

\begin{example}
Left-unimodal and right-unimodal permutations are
those in the geometric grid classes of the matrices
$$
M=
  \left( {\begin{array}{c}
   +1  \\
   -1 \\
  \end{array} } \right), 
   M=
  \left( {\begin{array}{c}
   -1  \\
   +1 \\
  \end{array} } \right),$$
respectively.
A drawing of a 
left-unimodal permutation
on the corresponding grid is shown in Figure~\ref{fig:grid_staircase}. 
\end{example}

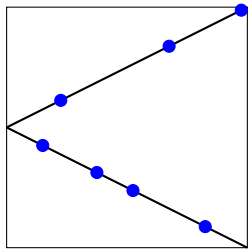
\begin{figure}[htb]
\centering
\begin{tikzpicture}[scale=0.8]
\draw (0,0) rectangle (4,4); \draw[thick] (0,2)--(4,4);
\draw[thick] (0,2)--(4,0);
\def\xa{6/10}
\def\xb{9/10}
\def\xc{15/10}
\def\xd{21/10}
\def\xe{27/10}
\def\xf{33/10}
\def\xg{39/10}
\def\ya{2-\xa/2}
\def\yc{2-\xc/2}
\def\yd{2-\xd/2}
\def\yf{2-\xf/2}
\def\yb{2+\xb/2}
\def\ye{2+\xe/2}
\def\yg{2+\xg/2}
\draw[fill,blue] (\xa,\ya) circle (0.1); \draw[fill,blue]
(\xb,\yb) circle (0.1); \draw[fill,blue] (\xc,\yc) circle (0.1);
\draw[fill,blue] (\xd,\yd) circle (0.1); \draw[fill,blue]
(\xe,\ye) circle (0.1); \draw[fill,blue] (\xf,\yf) circle (0.1);
\draw[fill,blue] (\xg,\yg) circle (0.1);
\end{tikzpicture}
\caption{A drawing of the permutation $4532617$ on the grid for left-unimodal permutations.}
\label{fig:grid_staircase}
\end{figure}

Being vertical rotations of left-unimodal permutations, arc permutations are precisely those that can be drawn on one of the grids in Figure~\ref{fig:arc}, which are obtained by vertically rotating the grid in Figure~\ref{fig:grid_staircase}. 
This pictorial grid description implies the following property.

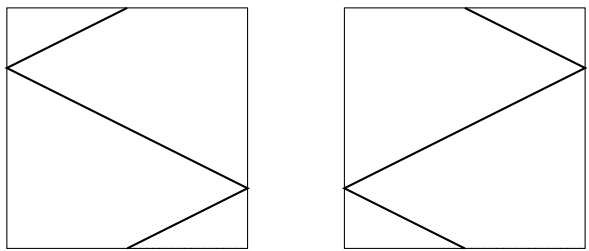
\begin{figure}[htb]
\centering
\begin{tikzpicture}[scale=0.8]
\draw (0,0) rectangle (4,4); \draw[dotted] (2,0)--(4,0);
\draw[thick] (2,0)--(4,1)--(0,3)--(2,4);
\end{tikzpicture}
\hspace{10mm}
\begin{tikzpicture}[scale=0.8]
\draw (0,0) rectangle (4,4); \draw[dotted] (2,0)--(4,0);
\draw[thick] (2,0)--(0,1)--(4,3)--(2,4);
\end{tikzpicture}
\caption{Grids for arc permutations.}\label{fig:arc}
\end{figure}

\begin{observation}\label{obs:arc}
    A permutation $\pi\in \symm_n$ with $\pi(j)=n$ is an arc permutation if and only if
    one of the following holds:
\begin{itemize}  
\item the prefix
    $\pi(1),\dots,\pi(j-1)$ is a left-unimodal sequence on $[j]$,
   and the suffix   $\pi(j+1),\dots,\pi(n)$ is a right-unimodal sequence on $[n-1]\setminus [j]$, or
\item   the prefix
    $\pi(1),\dots,\pi(j-1)$ is a left-unimodal sequence on  $[n-1]\setminus [n-j]$, and the suffix $\pi(j+1),\dots,\pi(n)$ is a right-unimodal sequence on $[n-j]$.
    \end{itemize}
\end{observation}

\subsection{A $\cDes$-preserving bijection} \label{sec:bijection}
Next we give a bijection between $D_{n-1,[i]}^{-1} C_n$ and $C_n D_{n-1,[i]}^{-1}$ that preserves the cyclic descent set. By allowing to $i$ range between $0$ and $n-2$ and using Equations~\eqref{eq:L-union} and~\eqref{eq:A-union}, we obtain a bijection between $\L_{n-1}C_n$ and $\A_n$. In fact, we will describe the bijection in this setting.

A key fact used in the construction below is that the  
descent set map $\Des$ is a bijection between $\L_{n-1}$ 
and the power set $2^{[n-2]}$. 

\begin{defn}\label{def:phi}
	Let $\phi:\ \L_{n-1} C_n\longrightarrow \A_n$ be the following map. Given  $\pi\in \L_{n-1} C_n$, let $j=\pi^{-1}(n)$  be the position of the letter $n$ in $\pi$, and construct $\phi\pi$ as follows:
	\begin{itemize}
		\item Let $(\phi \pi)(j)=n$.
		\item If $n\in \cDes(\pi)$, then the set of the leftmost $j-1$ entries in $\phi\pi$ is equal to $[j-1]$; otherwise, it is equal to $[n-1]\setminus[n-j]$.
		\item The order of the first $j-1$ entries in $\phi\pi$
		is given by the only left-unimodal permutation having descent set $\Des(\pi)\cap [j-2]$.
		\item The order of the last $n-j$ entries in $\phi\pi$ 
		is given by the only right-unimodal permutation having descent set $\{i-j:i\in\Des(\pi),i>j\}$.
		\end{itemize}	 
\end{defn}

\begin{example}\label{ex:phi}
	Take $43567281\in \L_{8}$ and $\pi=672819435\in \L_8 C_9$.
	Then $j=\pi^{-1}(9)=6$, so we set $(\phi\pi)(6)=9$. Since $9\not\in \cDes(\pi)$, the first  
	five entries in $\phi \pi$ are $[8]\setminus [3]=\{4,5,6,7,8\}$, and the last three entries are $\{1,2,3\}$.
	The order of the first five entries is given by the only left-unimodal permutation of $\{4,5,6,7,8\}$ with descent set $\Des(\pi)\cap [4]=\{2,4\}$, namely
	$67584$. 
	The order of the last three entries is given by the only right-unimodal permutation on $\{1,2,3\}$ with descent set $\{1\}$, namely
	$312$. Altogether, $\phi\pi= 675849312\in \A_9$. As expected, $\cDes(\phi\pi)=\{2,4,6,7\}=\cDes(\pi)$.
\end{example}

A pictorial description of the map $\phi$ is presented in Figure~\ref{fig:phi_arc}.
For $\pi \in \L_{n-1}  C_n$ with $\pi(j)=n$, the two cases correspond to $n\in \cDes(\pi)$ and $n\not\in \cDes(\pi)$, or  equivalently, to $\pi(1)<\pi(j+1)$ and $\pi(1)>\pi(j+1)$, respectively. Note that when $j=n$, we have $\phi\pi=\pi$.

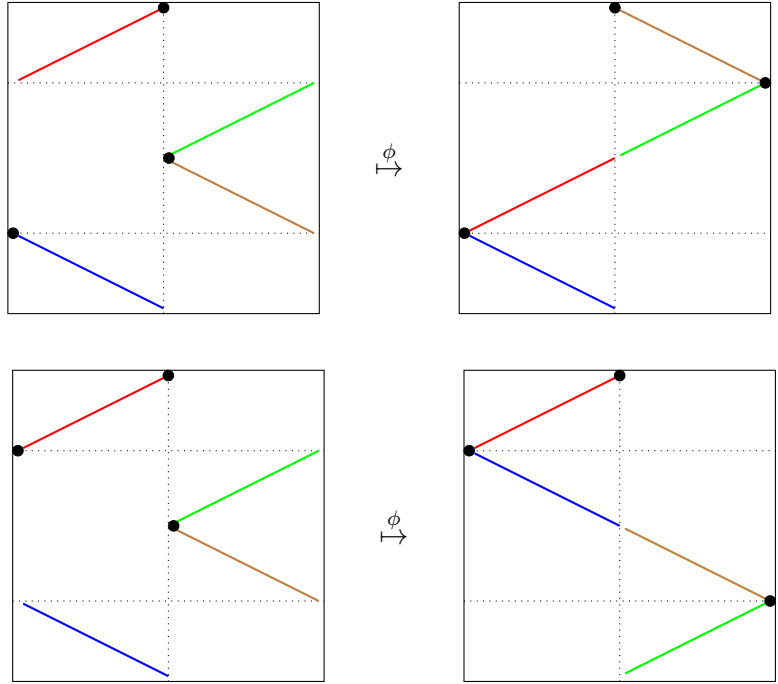
\begin{figure}
\def\eps{.07}
\begin{center}
	\begin{tikzpicture}[scale=1]
	\draw[thin,dotted] (2,-\eps)--(2,4+\eps);
	\draw[thin,dotted] (-\eps,1)--(4+\eps,1);
	\draw[thin,dotted] (-\eps,3)--(4+\eps,3);
	\draw (-\eps,-\eps) rectangle (4+\eps,4+\eps);
	\coordinate (a1) at (0,3);
	\coordinate (a2) at (2,4);
	\coordinate (b1) at (0,1);
	\coordinate (b2) at (2,0);
	\coordinate (c1) at (2,2);
	\coordinate (c2) at (4,3);
	\coordinate (d1) at (2,2);
	\coordinate (d2) at (4,1);
	\draw[thick,red] (a1)+(\eps,\eps/2)--(a2)+(-\eps,-\eps/2); 
	\draw[thick,blue] (b1)+(\eps,-\eps/2)--(b2)+(-\eps,\eps/2);
	\draw[thick,green] (c1)+(\eps,\eps/2)--(c2)+(-\eps,-\eps/2);
	\draw[thick,brown] (d1)+(\eps,-\eps/2)--(d2)+(-\eps,+\eps/2); 
	\draw[fill] (2,4) circle (\eps);
	\draw[fill] (0,1) circle (\eps);
	\draw[fill] (c1)+(\eps,0) circle (\eps);
	
	\draw (5,2) node {$\overset{\phi}{\mapsto}$};
	
	\begin{scope}[shift={(6,0)}]
	\draw[thin,dotted] (2,-\eps)--(2,4+\eps);
	\draw[thin,dotted] (-\eps,1)--(4+\eps,1);
	\draw[thin,dotted] (-\eps,3)--(4+\eps,3);
	\draw (-\eps,-\eps) rectangle (4+\eps,4+\eps);
	\coordinate (a1) at (0,1);
	\coordinate (a2) at (2,2);
	\coordinate (b1) at (0,1);
	\coordinate (b2) at (2,0);
	\coordinate (c1) at (2,2);
	\coordinate (c2) at (4,3);
	\coordinate (d1) at (2,4);
	\coordinate (d2) at (4,3);
	\draw[thick,red] (a1)+(\eps,\eps/2)--(a2)+(-\eps,-\eps/2); 
	\draw[thick,blue] (b1)+(\eps,-\eps/2)--(b2)+(-\eps,\eps/2);
	\draw[thick,green] (c1)+(\eps,\eps/2)--(c2)+(-\eps,-\eps/2);
	\draw[thick,brown] (d1)+(\eps,-\eps/2)--(d2)+(-\eps,+\eps/2); 
	\draw[fill] (2,4) circle (\eps);
	\draw[fill] (0,1) circle (\eps);
	\draw[fill] (c2) circle (\eps);
	\end{scope}
	\end{tikzpicture}
	\vspace{20pt}
	
	\begin{tikzpicture}[scale=1]
	\draw[thin,dotted] (2,-\eps)--(2,4+\eps);
	\draw[thin,dotted] (-\eps,1)--(4+\eps,1);
	\draw[thin,dotted] (-\eps,3)--(4+\eps,3);
	\draw (-\eps,-\eps) rectangle (4+\eps,4+\eps);
	\coordinate (a1) at (0,3);
	\coordinate (a2) at (2,4);
	\coordinate (b1) at (0,1);
	\coordinate (b2) at (2,0);
	\coordinate (c1) at (2,2);
	\coordinate (c2) at (4,3);
	\coordinate (d1) at (2,2);
	\coordinate (d2) at (4,1);
	\draw[thick,red] (a1)+(\eps,\eps/2)--(a2)+(-\eps,-\eps/2); 
	\draw[thick,blue] (b1)+(\eps,-\eps/2)--(b2)+(-\eps,\eps/2);
	\draw[thick,green] (c1)+(\eps,\eps/2)--(c2)+(-\eps,-\eps/2);
	\draw[thick,brown] (d1)+(\eps,-\eps/2)--(d2)+(-\eps,+\eps/2); 
	\draw[fill] (2,4) circle (\eps);
	\draw[fill] (0,3) circle (\eps);
	\draw[fill] (c1)+(\eps,0) circle (\eps);
	
	\draw (5,2) node {$\overset{\phi}{\mapsto}$};
	
	\begin{scope}[shift={(6,0)}]
	\draw[thin,dotted] (2,-\eps)--(2,4+\eps);
	\draw[thin,dotted] (-\eps,1)--(4+\eps,1);
	\draw[thin,dotted] (-\eps,3)--(4+\eps,3);
	\draw (-\eps,-\eps) rectangle (4+\eps,4+\eps);
	\coordinate (a1) at (0,3);
	\coordinate (a2) at (2,4);
	\coordinate (b1) at (0,3);
	\coordinate (b2) at (2,2);
	\coordinate (c1) at (2,0);
	\coordinate (c2) at (4,1);
	\coordinate (d1) at (2,2);
	\coordinate (d2) at (4,1);
	\draw[thick,red] (a1)+(\eps,\eps/2)--(a2)+(-\eps,-\eps/2); 
	\draw[thick,blue] (b1)+(\eps,-\eps/2)--(b2)+(-\eps,\eps/2);
	\draw[thick,green] (c1)+(\eps,\eps/2)--(c2)+(-\eps,-\eps/2);
	\draw[thick,brown] (d1)+(\eps,-\eps/2)--(d2)+(-\eps,+\eps/2); 
	\draw[fill] (2,4) circle (\eps);
	\draw[fill] (0,3) circle (\eps);
	\draw[fill] (c2) circle (\eps);
	\end{scope}
	\end{tikzpicture}
\end{center}
\caption{The $\cDes$-preserving bijection 
$\phi:\ \L_{n-1} C_n\longrightarrow \A_n$.}
\label{fig:phi_arc}
\end{figure}

\begin{figure}
\def\NE{(.333,.167)}
\def\SE{(.333,-.167)}
\centering
	\begin{tikzpicture}[scale=1]
	\draw[thin,dotted] (2,-\eps)--(2,4+\eps);
	\draw[thin,dotted] (-\eps,1)--(4+\eps,1);
	\draw[thin,dotted] (-\eps,3)--(4+\eps,3);
	\draw (-\eps,-\eps) rectangle (4+\eps,4+\eps);
	\coordinate (a1) at (0,3);
	\coordinate (a2) at (2,4);
	\coordinate (b1) at (0,1);
	\coordinate (b2) at (2,0);
	\coordinate (c1) at (2,2);
	\coordinate (c2) at (4,3);
	\coordinate (d1) at (2,2);
	\coordinate (d2) at (4,1);
	\draw[thick,red] (a1)+(\eps,\eps/2)--(a2)+(-\eps,-\eps/2); 
	\draw[thick,blue] (b1)+(\eps,-\eps/2)--(b2)+(-\eps,\eps/2);
	\draw[thick,green] (c1)+(\eps,\eps/2)--(c2)+(-\eps,-\eps/2);
	\draw[thick,brown] (d1)+(\eps,-\eps/2)--(d2)+(-\eps,+\eps/2); 
	\draw[fill] (2,4) circle (\eps);
	\draw[fill] (0,1) circle (\eps);
	\draw[fill] (c1)+(\eps,0) circle (\eps);
	
	\draw[fill,blue] ($(b1)!1/6!(b2)$) circle (\eps);
	\draw[fill,red] ($(a1)!1/3!(a2)$) circle (\eps);
	\draw[fill,red] ($(a1)!1/2!(a2)$) circle (\eps);
	\draw[fill,red] ($(a1)!2/3!(a2)$) circle (\eps);
	\draw[fill,blue] ($(b1)!5/6!(b2)$) circle (\eps);
	
	\draw[fill,brown] ($(d1)!1/7!(d2)$) circle (\eps);
	\draw[fill,green] ($(c1)!2/7!(c2)$) circle (\eps);
	\draw[fill,brown] ($(d1)!3/7!(d2)$) circle (\eps);
	\draw[fill,brown] ($(d1)!4/7!(d2)$) circle (\eps);
	\draw[fill,brown] ($(d1)!5/7!(d2)$) circle (\eps);
	\draw[fill,green] ($(c1)!6/7!(c2)$) circle (\eps);
	
	\draw (5,2) node {$\overset{\phi}{\mapsto}$};
	
	\begin{scope}[shift={(6,0)}]
	\draw[thin,dotted] (2,-\eps)--(2,4+\eps);
	\draw[thin,dotted] (-\eps,1)--(4+\eps,1);
	\draw[thin,dotted] (-\eps,3)--(4+\eps,3);
	\draw (-\eps,-\eps) rectangle (4+\eps,4+\eps);
	\coordinate (a1) at (0,1);
	\coordinate (a2) at (2,2);
	\coordinate (b1) at (0,1);
	\coordinate (b2) at (2,0);
	\coordinate (c1) at (2,2);
	\coordinate (c2) at (4,3);
	\coordinate (d1) at (2,4);
	\coordinate (d2) at (4,3);
	\draw[thick,red] (a1)+(\eps,\eps/2)--(a2)+(-\eps,-\eps/2); 
	\draw[thick,blue] (b1)+(\eps,-\eps/2)--(b2)+(-\eps,\eps/2);
	\draw[thick,green] (c1)+(\eps,\eps/2)--(c2)+(-\eps,-\eps/2);
	\draw[thick,brown] (d1)+(\eps,-\eps/2)--(d2)+(-\eps,+\eps/2); 
	\draw[fill] (2,4) circle (\eps);
	\draw[fill] (0,1) circle (\eps);
	\draw[fill] (c2) circle (\eps);
	
	\draw[fill,blue] ($(b1)!1/6!(b2)$) circle (\eps);
	\draw[fill,red] ($(a1)!1/3!(a2)$) circle (\eps);
	\draw[fill,red] ($(a1)!1/2!(a2)$) circle (\eps);
	\draw[fill,red] ($(a1)!2/3!(a2)$) circle (\eps);
	\draw[fill,blue] ($(b1)!5/6!(b2)$) circle (\eps);
	
	\draw[fill,brown] ($(d1)!1/7!(d2)$) circle (\eps);
	\draw[fill,green] ($(c1)!2/7!(c2)$) circle (\eps);
	\draw[fill,brown] ($(d1)!3/7!(d2)$) circle (\eps);
	\draw[fill,brown] ($(d1)!4/7!(d2)$) circle (\eps);
	\draw[fill,brown] ($(d1)!5/7!(d2)$) circle (\eps);
	\draw[fill,green] ($(c1)!6/7!(c2)$) circle (\eps);
	
	\end{scope}
	\end{tikzpicture}
\caption{Example: $\phi(3\,2\,11\,12\,13\,1\,14\,8\,7\,9\,6\,5\,4\,10)=3\,2\,4\,5\,6\,1\,14\,13\,7\,12\,11\,10\,8\,9.$}
\end{figure}

\begin{lemma}
The map $\phi$ is a $\cDes$-preserving bijection between $\L_{n-1} C_n$ and $\A_n$, satisfying $\pi(n)=(\phi\pi)(n)$. In particular,
	\[
	\sum\limits_{\pi\in \A_n}{\bf x}^{\cDes(\pi)}t^{\pi^{-1}(n)}= 	\sum\limits_{\pi\in \L_{n-1} C_n}{\bf x}^{\cDes(\pi)}t^{\pi^{-1}(n)}.
	\]
\end{lemma}

\begin{proof}
By construction, $\cDes(\phi\pi)=\cDes(\pi)$, and by Observation~\ref{obs:arc}, $\phi\pi\in \A_n$.
To prove that $\phi$ is a bijection, it suffices to describe its inverse map.
Define $\psi:\ \A_n \longrightarrow  \L_{n-1} C_n$ as follows. Given $\sigma\in \A_n$, let $j=\sigma^{-1}(n)$ be the position of $n$.
	\begin{itemize}
		\item[1.] Let $\hat \sigma= \sigma c_n^{-j}$. 
		Notice that $\hat{\sigma}(n)=n$ and $\cDes(\hat \sigma)=-j+\cDes(\sigma) \bmod n$.
		
		\item[2.] Let $\des(\hat\sigma)=|\Des(\hat\sigma)|$. Create a permutation $\bar\sigma$ by placing the decreasing sequence $\des(\hat\sigma),\dots, 2,1$ in positions $1+\Des(\hat\sigma)=\{i+1:\ \hat\sigma(i)>\hat\sigma(i+1)\}$, and by placing the increasing sequence $1+\des(\hat\sigma), \dots,n-1,n$ in the remaining positions. 
		Notice that  $\bar\sigma$ may be identified, by ignoring $\bar\sigma(n)=n$, with a permutation in $\L_{n-1}$.
		
		\item[3.] Let $\psi \sigma = \bar \sigma c^{j}$.
		\end{itemize}
It is easy to verify that $\psi$ and $\phi$ are inverses of each other.
\end{proof}

\begin{example}
Let $\sigma= 675849312\in \A_9$ be the permutation from Example~\ref{ex:phi}. Then $j=6$, $\hat\sigma=312675849$, and $-6+\cDes(\pi)=-6+\{2,4,6,7\}=\{1,5,7,9\}=\cDes(\hat\pi)$, with addition modulo $9$.
Since $\Des(\hat\sigma)=\{1,5,7\}$, we place the decreasing sequence $3,2,1$ in positions $1+\Des(\hat\sigma)=\{2,6,8\}$, obtaining ${*}3{*}{*}{*}2{*}1{*}$, and an increasing sequence in the remaining positions, obtaining $\bar\sigma= 435672819\in \L_{8}\subseteq\symm_9$. Finally, $\psi\sigma=\bar\sigma c^{j}=
6728194356$.
\end{example}

\subsection{From arc permutations to SYT} Next we give a bijective proof of Corollary~\ref{cor:main_arc}. The proof, which relies on the bijection from Section~\ref{sec:bijection}, provides an explicit set of SYT on which $\cDes$ has the same distribution as it has on arc permutations (as in Definition~\ref{def:cSp}).

\begin{theorem}\label{thm:arcSYT}
	For every $n\ge 1$,
	\[
	\sum\limits_{\pi\in \A_n}{\bf x}^{\cDes(\pi)}=
		\sum\limits_{k=0}^{n-2}\sum\limits_{T\in \SYT((n-k-1,1^k) \oplus (1))} {\bf x}^{\cDes(T)}.
		\]
\end{theorem}

\begin{proof}
Using Equation~\eqref{eq:A-union}, it suffices to show that for every $0\le k< n-1$,
		\[
		\sum\limits_{\pi\in C_n D_{n-1,[k]}^{-1}}{\bf x}^{\cDes(\pi)}=
			\sum\limits_{\pi\in D_{n-1,[k]}^{-1}C_n}{\bf x}^{\cDes(\pi)}=
			\sum\limits_{T\in \SYT((n-k-1,1^k) \oplus (1))} {\bf x}^{\cDes(T)}.
		\]
The left equality follows from Theorem~\ref{thm:equid-main1}.
To prove the right equality,
we construct a $\cDes$-preserving bijection $f$ between 
$D_{n-1,[k]}^{-1}C_n$ and $\SYT((n-k-1,1^k) \oplus (1))$.

Given $\sigma\in D_{n-1,[k]}^{-1}C_n$, write $\sigma=\tau c_n^{-j}$, where $\tau\in D_{n-1,[k]}^{-1}\subseteq \symm_n$ and $0\le j<n$. Let $f(\tau)$ be the SYT of shape $(n-k-1,1^k) \oplus (1)$ having entry $n$ in the northeast component, and whose first column consists of the entry $1$ and the elements in the set $1+\Des(\tau)$. Then let
\[
f(\sigma) = j+ f(\tau),
\]
with addition defined 
as in Example~\ref{ex2:cdes-SYT}.
Finally, apply the definition of $\cDes$ in Equation~\eqref{eq:cDes_near_hook} 
to verify that $f$ is a $\cDes$-preserving bijection.
\end{proof}

Composing the above bijection $f$ with the bijection
$\phi^{-1}: \A_n \to \L_{n-1} C_n$ from Definition~\ref{def:phi}, we obtain a $\cDes$-preserving bijection from $\A_n$ 
to the set 
$\bigsqcup_{k=0}^{n-2}  \SYT((n-k-1,1^k) \oplus (1))$.
An example is shown in Figure~\ref{fig:arc_to_SYT}.

\newcommand{\downmapsto}{\rotatebox[origin=c]{-90}{$\mapsto$}\mkern2mu}
\newcommand{\upmapsto}{\rotatebox[origin=c]{90}{$\mapsto$}\mkern2mu}

	\begin{figure}[htb]
\centering
		\def\eps{.08}
			\begin{tikzpicture}[scale=.8]
			
						\draw[thin,dotted] (2,-\eps)--(2,4+\eps);
						\draw[thin,dotted] (-\eps,1)--(4+\eps,1);
						\draw[thin,dotted] (-\eps,3)--(4+\eps,3);
						\draw (-\eps,-\eps) rectangle (4+\eps,4+\eps);
						\coordinate (a1) at (0,1);
						\coordinate (a2) at (2,2);
						\coordinate (b1) at (0,1);
						\coordinate (b2) at (2,0);
						\coordinate (c1) at (2,2);
						\coordinate (c2) at (4,3);
						\coordinate (d1) at (2,4);
						\coordinate (d2) at (4,3);
						\draw[thick,red] (a1)+(\eps,\eps/2)--(a2)+(-\eps,-\eps/2); 
						\draw[thick,blue] (b1)+(\eps,-\eps/2)--(b2)+(-\eps,\eps/2);
						\draw[thick,green] (c1)+(\eps,\eps/2)--(c2)+(-\eps,-\eps/2);
						\draw[thick,brown] (d1)+(\eps,-\eps/2)--(d2)+(-\eps,+\eps/2); 
						\draw[fill] (2,4) circle (\eps);
						\draw[fill] (0,1) circle (\eps);
											\draw[fill] (0.66,0.64) circle (\eps);
												\draw[fill] (1.33,1.66) circle (\eps);
												\draw[fill] (3,2.5) circle (\eps);
						\draw[fill] (c2) circle (\eps);
                        \draw (2,4.6) node {$\pi=213645\in C_6 D_{5,[2]}^{-1}\subseteq\A_6$};
			
			\draw (5,2) node {$\overset{\phi^{-1}}{\mapsto}$};
			
			\begin{scope}[shift={(6,0)}]
			
						\draw[thin,dotted] (2,-\eps)--(2,4+\eps);
						\draw[thin,dotted] (-\eps,1)--(4+\eps,1);
						\draw[thin,dotted] (-\eps,3)--(4+\eps,3);
						\draw (-\eps,-\eps) rectangle (4+\eps,4+\eps);
						\coordinate (a1) at (0,3);
						\coordinate (a2) at (2,4);
						\coordinate (b1) at (0,1);
						\coordinate (b2) at (2,0);
						\coordinate (c1) at (2,2);
						\coordinate (c2) at (4,3);
						\coordinate (d1) at (2,2);
						\coordinate (d2) at (4,1);
						\draw[thick,red] (a1)+(\eps,\eps/2)--(a2)+(-\eps,-\eps/2); 
						\draw[thick,blue] (b1)+(\eps,-\eps/2)--(b2)+(-\eps,\eps/2);
						\draw[thick,green] (c1)+(\eps,\eps/2)--(c2)+(-\eps,-\eps/2);
						\draw[thick,brown] (d1)+(\eps,-\eps/2)--(d2)+(-\eps,+\eps/2); 
						\draw[fill] (2,4) circle (\eps);
						\draw[fill] (0,1) circle (\eps);
						\draw[fill] (c1)+(\eps,0) circle (\eps);			
							\draw[fill] (0.66,0.64) circle (\eps);
							\draw[fill] (3,2.5) circle (\eps);
								\draw[fill] (1.33,3.66) circle (\eps);
                        \draw (2,4.6) node {$\sigma=215634\in D_{5,[2]}^{-1} C_6$};
			\end{scope}
			
						\draw (8,-.8) node {$\downmapsto\scriptstyle{\cdot c_6^{4}}$};
			
			\begin{scope}[shift={(6,-5.7)}]
						
						\draw[thin,dotted] (2,-\eps)--(2,4+\eps);
						\draw[thin,dotted] (-\eps,1)--(4+\eps,1);
						\draw[thin,dotted] (-\eps,3)--(4+\eps,3);
						\draw (-\eps,-\eps) rectangle (4+\eps,4+\eps);
						\coordinate (a1) at (2,3);
						\coordinate (a2) at (4,4);
						\coordinate (b1) at (2,1);
						\coordinate (b2) at (4,0);
						\coordinate (c1) at (0,2);
						\coordinate (c2) at (2,3);
						\coordinate (d1) at (0,2);
						\coordinate (d2) at (2,1);
						\draw[thick,red] (a1)+(\eps,\eps/2)--(a2)+(-\eps,-\eps/2); 
						\draw[thick,blue] (b1)+(\eps,-\eps/2)--(b2)+(-\eps,\eps/2);
						\draw[thick,green] (c1)+(\eps,\eps/2)--(c2)+(-\eps,-\eps/2);
						\draw[thick,brown] (d1)+(\eps,-\eps/2)--(d2)+(-\eps,+\eps/2); 
						\draw[fill] (4,4) circle (\eps);
						\draw[fill] (2,1) circle (\eps);
						\draw[fill] (c1)+(\eps,0) circle (\eps);
													\draw[fill] (2.66,0.64) circle (\eps);
													\draw[fill] (1,2.5) circle (\eps);
													\draw[fill] (3.33,3.66) circle (\eps);			
	                    \draw (2,-.5) node {$\tau=342156\in  D_{5,[2]}^{-1}$};
						
						\end{scope}

\draw (11,2) node {$\overset{f}{\mapsto}$};

\begin{scope}[shift={(12,-5.7)}]

\draw  (1,2) node  {$\young(:::6,125,3,4)$};

\end{scope}

\draw (11,-3.7) node {$\overset{f}{\mapsto}$};

\draw (13,-.8) node {$\upmapsto\scriptstyle{+4}$};

\begin{scope}[shift={(12,0)}]

\draw  (1,2) node  {$\young(:::4,136,2,5)$};
\draw (2,4.6) node {$f(\sigma)\in\SYT((3,1^2)\oplus(1))$};

\end{scope}

			\end{tikzpicture}
	\caption{An example of the $\cDes$-preserving bijections $C_6 D_{5,[2]}^{-1}\overset{\phi^{-1}}{\to}D_{5,[2]}^{-1}C_6\overset{f}{\to} \SYT((3,1^2)\oplus(1))$.} 

\label{fig:arc_to_SYT}	
	\end{figure}
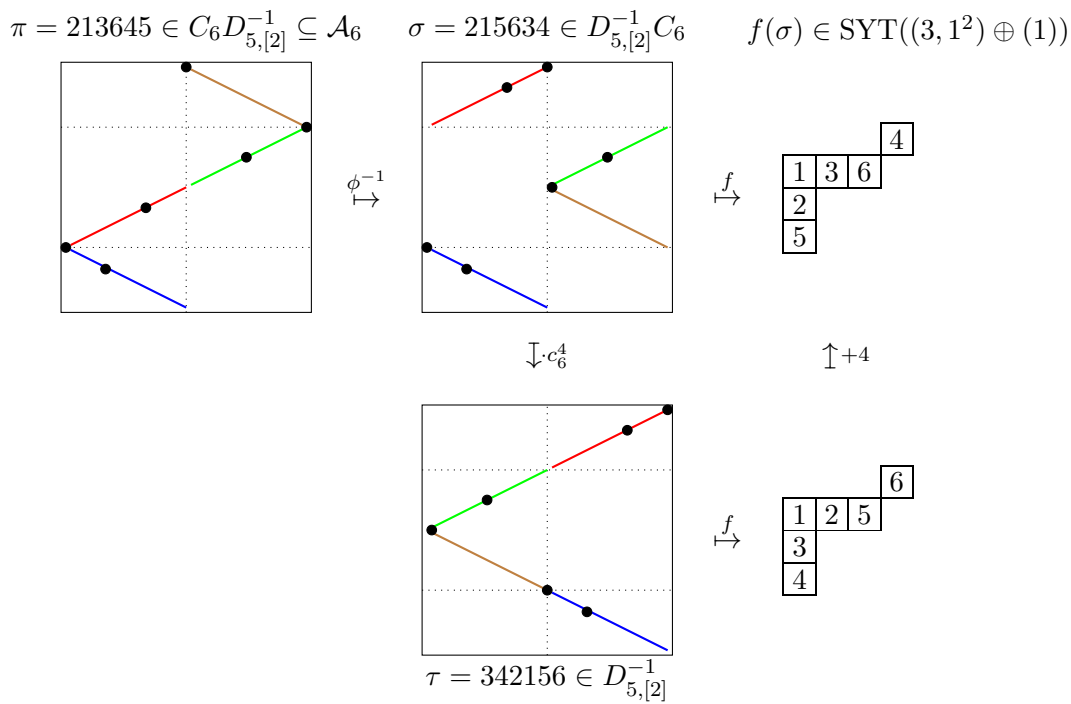

	\end{document}